\documentclass[11pt, reqno]{amsart}

\usepackage{amsmath, amsthm, amssymb, amsfonts}
\usepackage[margin=1.2in]{geometry}
\usepackage{hyperref}
\hypersetup{
    colorlinks=true,
    linkcolor=blue,
    citecolor=blue,
    urlcolor=blue
}

\usepackage[matrix,arrow,tips,curve]{xy}
\usepackage{array}
\usepackage{tikz}
\usepackage{dynkin-diagrams}
\usetikzlibrary{graphs, patterns, patterns.meta}
\usepackage{graphicx}
\usepackage{tabularx}
\usepackage{booktabs} 
\usepackage{wrapfig}
\usepackage{subfig}
\usepackage{enumerate}

\makeatletter
\pgfdeclarepatternformonly{sparse bold dots}
{\pgfqpoint{-1pt}{-1pt}} 
{\pgfqpoint{7pt}{7pt}}   
{\pgfqpoint{6pt}{6pt}}   
{
  \pgfsetlinewidth{0pt}
  \pgfpathcircle{\pgfqpoint{0pt}{0pt}}{1.2pt} 
  \pgfusepath{fill}
}

\pgfdeclarepatternformonly{sparse bold lines}
{\pgfqpoint{0pt}{0pt}}
{\pgfqpoint{10pt}{10pt}}
{\pgfqpoint{9pt}{9pt}}
{
  \pgfsetlinewidth{1.2pt}
  \pgfpathmoveto{\pgfqpoint{0pt}{0pt}}
  \pgfpathlineto{\pgfqpoint{9.1pt}{9.1pt}}
  \pgfusepath{stroke}
}
\makeatother

\theoremstyle{plain} 
\newtheorem{theorem}{Theorem}[section]
\newtheorem{conjecture}[theorem]{Conjecture}

\newtheorem{lemma}[theorem]{Lemma}
\newtheorem{proposition}[theorem]{Proposition}

\theoremstyle{definition}
\newtheorem{definition}[theorem]{Definition}
\newtheorem{example}[theorem]{Example}

\theoremstyle{remark}

\begin{document}

\title[Strong hull property for rank 3]{The strong hull property for affine irreducible Coxeter groups of rank 3}

\author{Ziming Liu}
\address{School of Mathematics\\
Shandong University\\
Jinan, Shandong, 250100 (China)}
\curraddr{Department of Mathematics\\
The Ohio State University\\
Columbus, OH 43210 (USA)}
\email{liu.12732@osu.edu}

\begin{abstract}
  A conjecture proposed by Gaetz and Gao asserts that the Cayley graph of any Coxeter group satisfies the strong hull property. In this paper, we prove this conjecture for all affine irreducible Coxeter groups of rank 3. Our approach exploits the geometry of Coxeter complexes to reduce the analysis of convex hulls to finitely many manageable configurations.
\end{abstract}

\maketitle

\section{Introduction}

Let $G$ be a connected undirected graph equipped with the distance function $d: V(G) \times V(G) \to \mathbb{Z}_+$. Specifically, $d(x,y)$ represents the shortest path length between vertices $x$ and $y$. A subset $C \subseteq V(G)$ is called \textit{convex} if for any $u, v \in C$ and any $w \in V(G)$ satisfying $d(u,w) + d(w,v) = d(u,v)$, the inclusion $w \in C$ necessarily holds. The \textit{convex hull} $\mathrm{Conv}(X)$ of a subset $X \subseteq V(G)$ is the minimal convex set containing $X$, equivalently expressed as the intersection of all convex sets containing $X$.

\begin{definition}\label{hull property and strong hull property}
    A graph $G$ is defined to satisfy the \textbf{hull property} if for any three vertices $u, v, w \in V(G)$, the cardinality inequality
\begin{equation}\label{hull property}
    |\mathrm{Conv}(u,v)| \cdot |\mathrm{Conv}(v,w)| \geq |\mathrm{Conv}(u,w)|
\end{equation}
holds. When the enhanced condition
\begin{equation}\label{strong hull property}
    |\mathrm{Conv}(u,v)| \cdot |\mathrm{Conv}(v,w)| \geq |\mathrm{Conv}(u,v,w)|
\end{equation}
holds, the graph $G$ is said to exhibit the \textbf{strong hull property}.
\end{definition}

Given a Coxeter group $W$, let $\mathrm{Cay}(W)$ denote its undirected right Cayley graph associated with its generating set. Gaetz-Gao \cite{gaetz2022hull} proposed Conjecture \ref{strong hull conj} concerning convexity properties:

\begin{conjecture}\label{strong hull conj}
    Every Coxeter group $W$ has the property that its Cayley graph $\mathrm{Cay}(W)$ satisfies the strong hull property.
\end{conjecture}

Gaetz-Gao \cite{gaetz2022hull} established the validity of Conjecture \ref{strong hull conj} for symmetric groups (type $A$), hyperoctahedral groups (type $B$), and all right-angled Coxeter groups. They further indicated that computational verification is feasible for some finite Coxeter groups, like types $D_4$, $F_4$, $G_2$, and $H_3$. The methodology for symmetric and hyperoctahedral groups employs insertion maps for linear extensions, which are combinatorial tools intrinsically connected to promotion operations \cite{schutzenberger1972promotion}. Notably, an independent confirmation for the symmetric group case was achieved by Chan-Pak-Panova \cite{chan2023effective}. Furthermore, Gaetz-Gao \cite{gaetz2022hull} developed a constructive approach for the case of right-angled Coxeter groups. From a structural perspective, right-angled Coxeter groups occupy opposed positions among Coxeter groups when compared with symmetric groups and hyperoctahedral groups. This dichotomy manifests algebraically through their non-commuting products $s_is_j$ possessing infinite order, a stark contrast to the small finite orders characterizing finite Coxeter groups. Furthermore, these groups constitute a fundamental object for hull metric verification due to their pervasive presence in geometric group theory \cite{dani2018large}.

To analyze the predictive strength of Conjecture~\ref{strong hull conj}, let's verify a specific restricted configuration. Consider an arbitrary permutation $\sigma$ in the symmetric group $S_n$, with $\sigma^{\mathrm{rev}}$ denoting its reverse permutation. The hull property yields that for any $2$-dimensional poset $P_\sigma$ associated with these permutations, we have
\begin{equation}\label{Sidorenko's result}
    e(P_\sigma) \cdot e(P_{\sigma^{\mathrm{rev}}}) \geq n!\,,
\end{equation}
where $e(P_\sigma)$ denotes the linear extension count, a result attributed to Sidorenko \cite{sidorenko1991inequalities}.

The initial demonstration of inequality \eqref{Sidorenko's result} by Sidorenko \cite{sidorenko1991inequalities} utilized max-flow min-cut techniques. Subsequent research has revealed deep connections between this inequality and diverse methodologies in convex geometry and combinatorial theory. Notably, Bollobás-Brightwell-Sidorenko \cite{BOLLOBAS1999329} provided an alternative convex geometric interpretation through partial results related to the unresolved Mahler Conjecture. More recently, Gaetz-Gao \cite{GAETZ2020107389, GAETZ2020101974} developed enhanced proofs incorporating the algebraic framework of \textit{generalized quotients} \cite{bjorner1988generalized} within the Coxeter group, thus establishing novel connections in this domain. The result in Gaetz-Gao \cite{gaetz2022hull} for the symmetric group extends Sidorenko's inequality \eqref{Sidorenko's result} to any pair of elements.

The following Theorem~\ref{main result} solves one class of cases of Conjecture \ref{strong hull conj}, namely for affine irreducible Coxeter groups of rank 3. In fact, this class only includes affine types $\widetilde{A}_2$, $\widetilde{C}_2$, and $\widetilde{G}_2$. The detailed explanation can be found in Section \ref{Background}.

\begin{theorem}\label{main result}
    Conjecture \ref{strong hull conj} holds for affine irreducible Coxeter groups of rank $3$.
\end{theorem}

To prove Theorem~\ref{main result}, we examine the geometric interpretations of the types $\widetilde{A}_2$, $\widetilde{C}_2$, and $\widetilde{G}_2$. Specifically, each can be represented as a triangulation of the two-dimensional Euclidean plane. We then analyze these triangular grids by mapping the Coxeter complexes onto the corresponding Cayley graphs. Utilizing classification and reduction techniques, we rigorously establish the results for these three cases through detailed computations in Section~\ref{proof of main}.

\section{Background}\label{Background}

\begin{definition}
    A \textit{Coxeter group} is a group $W$ together with a generating set $S=\{s_1,\cdots,s_r\}$ subject to the relations
\[
\left\{
\begin{array}{ll}
s_i^2=\mathbf{1} & \text{for } i=1,\cdots, r, \\
(s_i s_j)^{m_{ij}}=\mathbf{1} & \text{for } i\neq j \in \{1,\cdots, r\}
\end{array}
\right.
\]
where $m_{ii}=1$, otherwise $m_{ij}=m_{ji}\in\{2,3,\cdots,\infty\}$. One can also write it as a group presentation
$\langle s_1,\cdots, s_n\mid (s_i s_j)^{m_{ij}}=\mathbf{1}\rangle$. The elements of $S$ are called \textit{Coxeter generators} and the cardinality of $S$ is called the \textit{rank} of the \textit{Coxeter system} $(W, S)$.
\end{definition}

There are several ways to describe a Coxeter group. Consider the following mapping 
\[\begin{aligned}
    m:S\times S &\longrightarrow \mathbb{Z}_+\sqcup\{\infty\}\\
    (s_i,s_j)&\longmapsto m_{ij}.
\end{aligned}\]
It can be represented by a graph whose vertices are the elements of $S$ and attach $s_i$ and $s_j$ to form an edge if $m(s_i,s_j)\geq 3$. Label the edges with $m_{ij}$ where $m(s_i,s_j)\geq 4$. The resulting graph is the \textit{Coxeter graph}. A Coxeter system is \textit{irreducible} if its Coxeter graph is connected.

Coxeter groups were classified in 1935 for the finite case \cite{coxeter1935complete}. The aim of this paper is not on the case of finite Coxeter groups. Instead, we focus on affine Coxeter groups. The irreducible affine Coxeter groups correspond to positive semi-definite Coxeter matrices and can be fully classified by their connected Coxeter graphs, as listed in Tab.~\ref{Affine irreducible Coxeter groups}.

\begin{table}[!h]
\caption{Affine irreducible Coxeter groups}
\label{Affine irreducible Coxeter groups}
\centering
\renewcommand{\arraystretch}{1.5} 
\setlength{\tabcolsep}{10pt} 
\resizebox{\textwidth}{!}{ 
\begin{tabular}{|>{\centering\arraybackslash}m{2.5cm}|>{\centering\arraybackslash}m{7cm}|>{\centering\arraybackslash}m{2.5cm}|>{\centering\arraybackslash}m{7cm}|}
\hline
\textbf{Type} & \textbf{Graph} & \textbf{Type} & \textbf{Graph} \\ \hline
$\widetilde{A}_1=I_2(\infty)$ & \begin{tikzpicture}[shorten >=1pt, node distance=2cm, auto]

    \node (1) at (0, 0) [circle, fill=black, inner sep=1.5pt] {};
    
    \node (2) at (1, 0) [circle, fill=black, inner sep=1.5pt] {};

    \draw (1) -- (2) node[midway, above] {$\infty$};

\end{tikzpicture} & $\widetilde{E}_6$ & \begin{tikzpicture}[shorten >=1pt, node distance=2cm, auto]

    \node (1) at (0, 0) [circle, fill=black, inner sep=1.5pt] {};
    
    \node (2) at (1, 0) [circle, fill=black, inner sep=1.5pt] {};
    
    \node (3) at (2, 0) [circle, fill=black, inner sep=1.5pt] {};
    
    \node (4) at (3, 0) [circle, fill=black, inner sep=1.5pt] {};

    \node (5) at (4, 0) [circle, fill=black, inner sep=1.5pt] {};
    
    \node (6) at (2, 1) [circle, fill=black, inner sep=1.5pt] {};
    
    \node (7) at (2, 2) [circle, fill=black, inner sep=1.5pt] {};

    \draw (1) -- (2);
    \draw (2) -- (3);
    \draw (3) -- (4);
    \draw (4) -- (5);
    \draw (3) -- (6);
    \draw (6) -- (7);

\end{tikzpicture} \\ \hline
$\widetilde{A}_{n-1}$, $n\geq 3$ & \begin{tikzpicture}[shorten >=1pt, node distance=2cm, auto]

    \node (1) at (0, 0) [circle, fill=black, inner sep=1.5pt] {};
    \node at (0, -0.4) {$s_1$};
    
    \node (2) at (1, 0) [circle, fill=black, inner sep=1.5pt] {};
    \node at (1, -0.4) {$s_2$};
    
    \node (3) at (3, 0) [circle, fill=black, inner sep=1.5pt] {};
    \node at (3, -0.4) {$s_{n-2}$};
    
    \node (4) at (4, 0) [circle, fill=black, inner sep=1.5pt] {};
    \node at (4, -0.4) {$s_{n-1}$};

    \node (5) at (2, 1) [circle, fill=black, inner sep=1.5pt] {};
    \node at (1.6, 1) {$s_n$};

    \draw (1) -- (2);
    \draw (2) -- (3) [dashed];  
    \draw (3) -- (4);
    \draw (1) -- (5);
    \draw (4) -- (5);

\end{tikzpicture} & $\widetilde{E}_7$ & \begin{tikzpicture}[shorten >=1pt, node distance=2cm, auto]

    \node (1) at (0, 0) [circle, fill=black, inner sep=1.5pt] {};
    
    \node (2) at (1, 0) [circle, fill=black, inner sep=1.5pt] {};
    
    \node (3) at (2, 0) [circle, fill=black, inner sep=1.5pt] {};
    
    \node (4) at (3, 0) [circle, fill=black, inner sep=1.5pt] {};

    \node (5) at (4, 0) [circle, fill=black, inner sep=1.5pt] {};
    
    \node (6) at (5, 0) [circle, fill=black, inner sep=1.5pt] {};
    
    \node (7) at (6, 0) [circle, fill=black, inner sep=1.5pt] {};

    \node (8) at (3, 1) [circle, fill=black, inner sep=1.5pt] {};

    \draw (1) -- (2);
    \draw (2) -- (3);
    \draw (3) -- (4);
    \draw (4) -- (5);
    \draw (5) -- (6);
    \draw (6) -- (7);
    \draw (4) -- (8);

\end{tikzpicture} \\ \hline
$\widetilde{B}_n$, $n\geq 3$ & \begin{tikzpicture}[shorten >=1pt, node distance=2cm, auto]

    \node (1) at (0, 0) [circle, fill=black, inner sep=1.5pt] {};
    \node at (0, -0.4) {$s_0$};
    
    \node (2) at (1, 0) [circle, fill=black, inner sep=1.5pt] {};
    \node at (1, -0.4) {$s_1$};
    
    \node (3) at (2, 0) [circle, fill=black, inner sep=1.5pt] {};
    \node at (2, -0.4) {$s_2$};
    
    \node (4) at (4, 0) [circle, fill=black, inner sep=1.5pt] {};
    \node at (4, -0.4) {$s_{n-2}$};

    \node (5) at (5, 0) [circle, fill=black, inner sep=1.5pt] {};
    \node at (5, -0.4) {$s_{n-1}$};

    \node (6) at (4, 1) [circle, fill=black, inner sep=1.5pt] {};
    \node at (3.6, 1) {$s_n$};

    \draw (1) -- (2) node[midway, above] {$4$};
    \draw (2) -- (3);
    \draw (3) -- (4) [dashed];  
    \draw (4) -- (5);
    \draw (4) -- (6);

\end{tikzpicture} & $\widetilde{E}_8$ & \begin{tikzpicture}[shorten >=1pt, node distance=2cm, auto]

    \node (1) at (0, 0) [circle, fill=black, inner sep=1.5pt] {};
    
    \node (2) at (1, 0) [circle, fill=black, inner sep=1.5pt] {};
    
    \node (3) at (2, 0) [circle, fill=black, inner sep=1.5pt] {};
    
    \node (4) at (3, 0) [circle, fill=black, inner sep=1.5pt] {};

    \node (5) at (4, 0) [circle, fill=black, inner sep=1.5pt] {};
    
    \node (6) at (5, 0) [circle, fill=black, inner sep=1.5pt] {};
    
    \node (7) at (6, 0) [circle, fill=black, inner sep=1.5pt] {};

    \node (8) at (7, 0) [circle, fill=black, inner sep=1.5pt] {};

    \node (9) at (2, 1) [circle, fill=black, inner sep=1.5pt] {};

    \draw (1) -- (2);
    \draw (2) -- (3);
    \draw (3) -- (4);
    \draw (4) -- (5);
    \draw (5) -- (6);
    \draw (6) -- (7);
    \draw (7) -- (8);
    \draw (3) -- (9);

\end{tikzpicture} \\ \hline
$\widetilde{C}_n$, $n\geq 2$ & \begin{tikzpicture}[shorten >=1pt, node distance=2cm, auto]

    \node (1) at (0, 0) [circle, fill=black, inner sep=1.5pt] {};
    \node at (0., -0.4) {$s_0$};
    
    \node (2) at (1, 0) [circle, fill=black, inner sep=1.5pt] {};
    \node at (1, -0.4) {$s_1$};
    
    \node (3) at (2, 0) [circle, fill=black, inner sep=1.5pt] {};
    \node at (2, -0.4) {$s_2$};
    
    \node (4) at (4, 0) [circle, fill=black, inner sep=1.5pt] {};
    \node at (4, -0.4) {$s_{n-2}$};

    \node (5) at (5, 0) [circle, fill=black, inner sep=1.5pt] {};
    \node at (5, -0.4) {$s_{n-1}$};

    \node (6) at (6, 0) [circle, fill=black, inner sep=1.5pt] {};
    \node at (6, -0.4) {$s_n$};

    \draw (1) -- (2) node[midway, above] {$4$};
    \draw (2) -- (3);
    \draw (3) -- (4) [dashed];  
    \draw (4) -- (5);
    \draw (5) -- (6) node[midway, above] {$4$};

\end{tikzpicture}  & $\widetilde{F}_4$ & \begin{tikzpicture}[shorten >=1pt, node distance=2cm, auto]

    \node (1) at (0, 0) [circle, fill=black, inner sep=1.5pt] {};
    
    \node (2) at (1, 0) [circle, fill=black, inner sep=1.5pt] {};
    
    \node (3) at (2, 0) [circle, fill=black, inner sep=1.5pt] {};
    
    \node (4) at (3, 0) [circle, fill=black, inner sep=1.5pt] {};

    \node (5) at (4, 0) [circle, fill=black, inner sep=1.5pt] {};

    \draw (1) -- (2);
    \draw (2) -- (3) node[midway, above] {$4$};
    \draw (3) -- (4);
    \draw (4) -- (5);

\end{tikzpicture} \\ \hline
$\widetilde{D}_n$, $n\geq 4$ & \begin{tikzpicture}[shorten >=1pt, node distance=2cm, auto]

    \node (1) at (0, 0) [circle, fill=black, inner sep=1.5pt] {};
    \node at (0, -0.4) {$s_1$};
    
    \node (2) at (1, 0) [circle, fill=black, inner sep=1.5pt] {};
    \node at (1, -0.4) {$s_2$};
    
    \node (3) at (2, 0) [circle, fill=black, inner sep=1.5pt] {};
    \node at (2, -0.4) {$s_3$};
    
    \node (4) at (4, 0) [circle, fill=black, inner sep=1.5pt] {};
    \node at (4, -0.4) {$s_{n-2}$};

    \node (5) at (5, 0) [circle, fill=black, inner sep=1.5pt] {};
    \node at (5, -0.4) {$s_{n-1}$};

    \node (6) at (4, 1) [circle, fill=black, inner sep=1.5pt] {};
    \node at (3.6, 1) {$s_n$};

    \node (7) at (1, 1) [circle, fill=black, inner sep=1.5pt] {};
    \node at (0.6, 1) {$s_0$};

    \draw (1) -- (2);
    \draw (2) -- (3);
    \draw (3) -- (4) [dashed];  
    \draw (4) -- (5);
    \draw (4) -- (6);
    \draw (2) -- (7);

\end{tikzpicture} & $\widetilde{G}_2$ & \begin{tikzpicture}[shorten >=1pt, node distance=2cm, auto]

    \node (1) at (0, 0) [circle, fill=black, inner sep=1.5pt] {};
    
    \node (2) at (1, 0) [circle, fill=black, inner sep=1.5pt] {};
    
    \node (3) at (2, 0) [circle, fill=black, inner sep=1.5pt] {};

    \draw (1) -- (2) node[midway, above] {$6$};
    \draw (2) -- (3);

\end{tikzpicture} \\ \hline
\end{tabular}
}
\end{table}

Tab.~\ref{Affine irreducible Coxeter groups} is a list of affine irreducible Coxeter groups. Prop.~A.~17 of Malle-Testerman \cite{malle2011linear} and Section~6.7 in Humphreys \cite{humphreys1990reflection} imply that the three types of affine irreducible groups of rank $3$ are $\widetilde{A}_2$, $\widetilde{C}_2$, and $\widetilde{G}_2$.

To study the metric properties of a Coxeter group $W$ with generating set $S$, we naturally consider its \textit{Cayley graph} $\mathrm{Cay}(W,S)$. The vertices of $\mathrm{Cay}(W,S)$ correspond to the elements of $W$, and an edge connects $u$ and $v$ if $u^{-1}v \in S$. The shortest path metric on this graph defines the distance $d(u,v)$ between any two elements. The \textit{length} of an element $w \in W$ is denoted by $\ell(w) = d(\mathbf{1}, w)$. It immediately follows that $d(u, v) = \ell(u^{-1} v)$.

Geometrically, the Coxeter group $W$ acts simply transitively on a geometric structure known as the \textit{Coxeter complex}. The maximal simplices of this complex are called \textit{chambers}, which are in bijection with the elements of $W$. Two chambers are adjacent and share a codimension-$1$ face (referred to as a \textit{panel}) if their corresponding group elements are connected by an edge in $\mathrm{Cay}(W,S)$. 

A path in the Cayley graph naturally corresponds to a \textit{gallery} in the Coxeter complex, which is a finite sequence of chambers $(c_0, c_1, \dots, c_k)$ where each adjacent pair $c_{j-1}$ and $c_j$ shares a panel. A gallery from $x$ to $y$ is \textit{minimal} if its length $k$ equals the distance $d(x, y)$ in the Cayley graph.

A \textit{reflection} $r$ in $W$ is a conjugate of some generator $s \in S$. The fixed points of $r$ acting on the Coxeter complex form a \textit{wall} $M_r$, which is a subcomplex of codimension $1$. A panel lies in $M_r$ if and only if its two adjacent chambers are interchanged by $r$. Consequently, a gallery $(c_0, \dots, c_k)$ is said to \textit{cross} the wall $M_r$ if the reflection $r$ interchanges $c_{i-1}$ and $c_i$ for some $1 \leq i \leq k$.

The proof of the following fundamental lemma regarding minimal galleries and walls can be found in Ronan \cite{ronan2009lectures}.

\begin{lemma}\label{cannot cross wall twice}
    \begin{enumerate}[(i)]
        \item If $y$ is adjacent to $y'$ and distinct from it, then $d(x, y') = d(x, y) \pm 1$.
        \item A minimal gallery (i.e., a shortest path in the Cayley graph) cannot cross any wall more than once.
    \end{enumerate}
\end{lemma}

Proposition~\ref{union is convex hull} is an immediate consequence of Lemma~\ref{cannot cross wall twice}.

\begin{proposition}\label{union is convex hull}
    The union of all minimal galleries whose endpoints are the chambers $u$ and $v$ precisely forms the convex hull $\mathrm{Conv}(u,v)$ in the Cayley graph.
\end{proposition}

For an affine Coxeter group, its Coxeter complex can be represented as a regular triangulation of a Euclidean space, where all chambers are isometric Euclidean simplices and the walls are Euclidean hyperplanes.

\begin{example}[Coxeter complexes]\label{Coxeter complexes as buildings}
In the case of affine irreducible Coxeter groups of rank $3$, each chamber is a $2$-dimensional Euclidean triangle. For any generators $i, j \in I$, the angle between the $i$-face and the $j$-face of the chamber is exactly $\frac{\pi}{m_{ij}}$. Since the interior angles of a Euclidean triangle must sum to $\pi$, the matrix entries must satisfy the geometric constraint
$$
\frac{1}{m_{12}} + \frac{1}{m_{23}} + \frac{1}{m_{31}} = 1.
$$
The only Coxeter graphs that yield integer solutions to this equation correspond to the types $\widetilde{A}_2$, $\widetilde{C}_2$, and $\widetilde{G}_2$. Therefore, the geometric structures of these groups are realized as specific regular triangulations of the two-dimensional Euclidean plane. Fig.~\ref{fig:reflection hyperplanes in Euclidean plane} illustrates the corresponding planar tessellations for these three affine types, which will serve as the primary geometric models for our convex hull analysis.

\begin{figure}[htbp]
    \centering
    \begin{minipage}[b]{0.32\textwidth}
        \centering
        \subfloat[Type $\widetilde{A}_2$]{\resizebox{\linewidth}{!}{\begin{tikzpicture}[shorten >=1pt, baseline=(current bounding box.center)]
            \node (1) at (0, 0) {};
            \node (11) at (3, {3*sqrt(3)}) {};
            \node (2) at (1, 0) {};
            \node (12) at (4, {3*sqrt(3)}) {};
            \node (3) at (2, 0) {};
            \node (13) at (5, {3*sqrt(3)}) {};
            \node (4) at (3, 0) {};
            \node (14) at (6, {3*sqrt(3)}) {};
            \node (5) at (4, 0) {};
            \node (23) at (6, {2*sqrt(3)}) {};
            \node (6) at (5, 0) {};
            \node (21) at (6, {sqrt(3)}) {};
            \node (7) at (6, 0) {};
            \node (8) at (0, {3*sqrt(3)}) {};
            \node (9) at (1, {3*sqrt(3)}) {};
            \node (10) at (2, {3*sqrt(3)}) {};
            \node (20) at (6, {0.5*sqrt(3)}) {};
            \node (22) at (6, {1.5*sqrt(3)}) {};
            \node (24) at (6, {2.5*sqrt(3)}) {};
            \node (15) at (0, {0.5*sqrt(3)}) {};
            \node (16) at (0, {sqrt(3)}) {};
            \node (17) at (0, {1.5*sqrt(3)}) {};
            \node (18) at (0, {2*sqrt(3)}) {};
            \node (19) at (0, {2.5*sqrt(3)}) {};

            \draw (1) -- (11);
            \draw (2) -- (12);
            \draw (3) -- (13);
            \draw (4) -- (14);
            \draw (5) -- (23);
            \draw (6) -- (21);
            \draw (7) -- (11);
            \draw (4) -- (8);
            \draw (5) -- (9);
            \draw (6) -- (10);
            \draw (3) -- (18);
            \draw (2) -- (16);
            \draw (16) -- (10);
            \draw (9) -- (18);
            \draw (21) -- (12);
            \draw (23) -- (13);
            \draw (15) -- (20);
            \draw (16) -- (21);
            \draw (17) -- (22);
            \draw (18) -- (23);
            \draw (19) -- (24);
        \end{tikzpicture}}}
    \end{minipage}%
    \hspace{0.01\textwidth}%
    \begin{minipage}[b]{0.32\textwidth}
        \centering
        \subfloat[Type $\widetilde{C}_2$]{\resizebox{\linewidth}{!}{\begin{tikzpicture}[shorten >=1pt, baseline=(current bounding box.center)]
            \node (1) at (0, 0) {};
            \node (2) at (1, 0) {};
            \node (3) at (2, 0) {};
            \node (4) at (3, 0) {};
            \node (5) at (4, 0) {};
            \node (6) at (5, 0) {};
            \node (7) at (6, 0) {};
            \node (8) at (0, 5) {};
            \node (9) at (1, 5) {};
            \node (10) at (2, 5) {};
            \node (11) at (3, 5) {};
            \node (12) at (4, 5) {};
            \node (13) at (5, 5) {};
            \node (14) at (6, 5) {};
            \node (15) at (0, 1) {};
            \node (16) at (0, 2) {};
            \node (17) at (0, 3) {};
            \node (18) at (0, 4) {};
            \node (19) at (6, 1) {};
            \node (20) at (6, 2) {};
            \node (21) at (6, 3) {};
            \node (22) at (6, 4) {};

            \draw (15) -- (19);
            \draw (16) -- (20);
            \draw (17) -- (21);
            \draw (18) -- (22);
            \draw (2) -- (9);
            \draw (3) -- (10);
            \draw (4) -- (11);
            \draw (5) -- (12);
            \draw (6) -- (13);
            \draw (1) -- (13);
            \draw (11) -- (16);
            \draw (3) -- (22);
            \draw (5) -- (20);
            \draw (3) -- (16);
            \draw (5) -- (18);
            \draw (7) -- (9);
            \draw (11) -- (20);
            \draw (9) -- (18);
            \draw (13) -- (22);
        \end{tikzpicture}}}
    \end{minipage}%
    \hspace{0.01\textwidth}%
    \begin{minipage}[b]{0.32\textwidth}
        \centering
        \subfloat[Type $\widetilde{G}_2$]{\resizebox{\linewidth}{!}{\begin{tikzpicture}[shorten >=1pt, baseline=(current bounding box.center)]\label{reflection hyperplane of G2}
            \node (1) at (0, 0) {};
            \node (11) at (3, {3*sqrt(3)}) {};
            \node (2) at (1, 0) {};
            \node (12) at (4, {3*sqrt(3)}) {};
            \node (3) at (2, 0) {};
            \node (13) at (5, {3*sqrt(3)}) {};
            \node (4) at (3, 0) {};
            \node (14) at (6, {3*sqrt(3)}) {};
            \node (5) at (4, 0) {};
            \node (23) at (6, {2*sqrt(3)}) {};
            \node (6) at (5, 0) {};
            \node (21) at (6, {sqrt(3)}) {};
            \node (7) at (6, 0) {};
            \node (8) at (0, {3*sqrt(3)}) {};
            \node (9) at (1, {3*sqrt(3)}) {};
            \node (10) at (2, {3*sqrt(3)}) {};
            \node (20) at (6, {0.5*sqrt(3)}) {};
            \node (22) at (6, {1.5*sqrt(3)}) {};
            \node (24) at (6, {2.5*sqrt(3)}) {};
            \node (15) at (0, {0.5*sqrt(3)}) {};
            \node (16) at (0, {sqrt(3)}) {};
            \node (17) at (0, {1.5*sqrt(3)}) {};
            \node (18) at (0, {2*sqrt(3)}) {};
            \node (19) at (0, {2.5*sqrt(3)}) {};
            \node (25) at (1.5, 0) {};
            \node (26) at (1.5, {3*sqrt(3)}) {};
            \node (27) at (4.5, 0) {};
            \node (28) at (4.5, {3*sqrt(3)}) {};

            \draw (1) -- (11);
            \draw (2) -- (12);
            \draw (3) -- (13);
            \draw (4) -- (14);
            \draw (5) -- (23);
            \draw (6) -- (21);
            \draw (7) -- (11);
            \draw (4) -- (8);
            \draw (5) -- (9);
            \draw (6) -- (10);
            \draw (3) -- (18);
            \draw (2) -- (16);
            \draw (16) -- (10);
            \draw (9) -- (18);
            \draw (21) -- (12);
            \draw (23) -- (13);
            \draw (15) -- (20);
            \draw (16) -- (21);
            \draw (17) -- (22);
            \draw (18) -- (23);
            \draw (19) -- (24);
            \draw (25) -- (26);
            \draw (27) -- (28);
            \draw (11) -- (18);
            \draw (14) -- (16);
            \draw (1) -- (23);
            \draw (4) -- (21);
            \draw (4) -- (16);
            \draw (7) -- (18);
            \draw (8) -- (21);
            \draw (11) -- (23);
            \draw (4) -- (11);
        \end{tikzpicture}}}
    \end{minipage}
    \caption{The triangulation of Euclidean space for rank-$3$ affine irreducible Coxeter groups}
    \label{fig:reflection hyperplanes in Euclidean plane}
\end{figure}
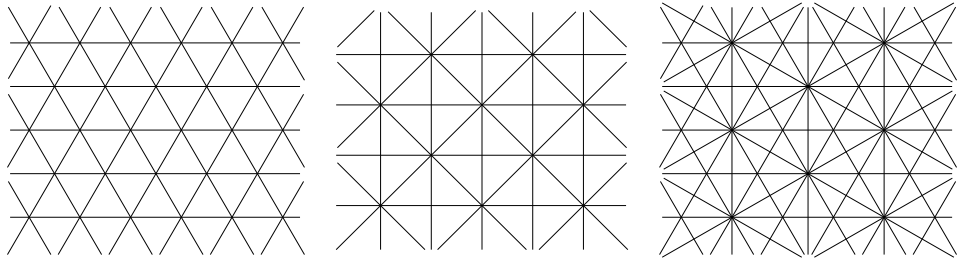
\end{example}

\begin{example}[The infinite dihedral group]\label{infinite dihedral group}
    Before analyzing the planar geometries of rank $3$ groups, we illustrate these metric concepts using the rank $2$ affine group $\widetilde{A}_1$, also known as the \textit{infinite dihedral group} $I_2(\infty)$ (see Tab.~\ref{Affine irreducible Coxeter groups}). Its Coxeter complex can be naturally viewed as a triangulation of the $1$-dimensional Euclidean space (the real line). 
    Let $r_1$ and $r_2$ be reflections on the real line. If we consider the point of reflection for $r_1$ as $x=0$ and for $r_2$ as $x=1$ (as shown in Fig.~\ref{fig:reflections generate infinite dihedral group}), we can express them as functions on $\mathbb{R}$:
    \[r_1(x)=-x,\quad r_2(x)=2-x.\]
    It is easy to check that the reflections $r_1$ and $r_2$ generate $I_2(\infty)$.

    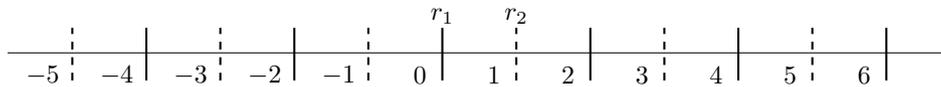
\begin{figure}[htbp]
    \centering
    \resizebox{\textwidth}{!}{

    \begin{tikzpicture}[shorten >=1pt, node distance=2cm, auto]
            \node (1) at (0, 0) {};
            \node (2) at (13, 0) {};
            \node (3) at (1, -0.5) {};
            \node (4) at (1, 0.5) {};
            \node at (0.6, -0.3) {$-5$};
            \node (5) at (2, -0.5) {};
            \node (6) at (2, 0.5) {};
            \node at (1.6, -0.3) {$-4$};
            \node (7) at (3, -0.5) {};
            \node (8) at (3, 0.5) {};
            \node at (2.6, -0.3) {$-3$};
            \node (9) at (4, -0.5) {};
            \node (10) at (4, 0.5) {};
            \node at (3.6, -0.3) {$-2$};
            \node (11) at (5, -0.5) {};
            \node (12) at (5, 0.5) {};
            \node at (4.6, -0.3) {$-1$};
            \node (13) at (6, -0.5) {};
            \node (14) at (6, 0.5) {};
            \node at (5.7, -0.3) {$0$};
            \node (15) at (7, -0.5) {};
            \node (16) at (7, 0.5) {};
            \node at (6.7, -0.3) {$1$};
            \node (17) at (8, -0.5) {};
            \node (18) at (8, 0.5) {};
            \node at (7.7, -0.3) {$2$};
            \node (19) at (9, -0.5) {};
            \node (20) at (9, 0.5) {};
            \node at (8.7, -0.3) {$3$};
            \node (21) at (10, -0.5) {};
            \node (22) at (10, 0.5) {};
            \node at (9.7, -0.3) {$4$};
            \node (23) at (11, -0.5) {};
            \node (24) at (11, 0.5) {};
            \node at (10.7, -0.3) {$5$};
            \node (25) at (12, -0.5) {};
            \node (26) at (12, 0.5) {};
            \node at (11.7, -0.3) {$6$};
            \node at (6, 0.5) {$r_1$};
            \node at (7, 0.5) {$r_2$};
            
            \draw (1) -- (2);
            \draw[thick, dashed] (3) -- (4);
            \draw[thick] (5) -- (6);
            \draw[thick, dashed] (7) -- (8);
            \draw[thick] (9) -- (10);
            \draw[thick, dashed] (11) -- (12);
            \draw[thick] (13) -- (14);
            \draw[thick, dashed] (15) -- (16);
            \draw[thick] (17) -- (18);
            \draw[thick, dashed] (19) -- (20);
            \draw[thick] (21) -- (22);
            \draw[thick, dashed] (23) -- (24);
            \draw[thick] (25) -- (26);
            
        \end{tikzpicture}}
    \caption{The reflections generate $I_2(\infty)$}
    \label{fig:reflections generate infinite dihedral group}
    \end{figure}

    We can use the drawing trick from Remark~1.49 in Meier \cite{JohnMeier.2008.Groups} to construct the Cayley graph for $I_2(\infty)$ with respect to the generators $r_1$ and $r_2$ (as in Fig.~\ref{fig:Cayley graph infinite dihedral group}). The group $I_2(\infty)$ acts on the real line $\mathbb{R}$, where $r_1$ corresponds to the reflection fixing $0$, and $r_2$ corresponds to the reflection fixing $1$.

        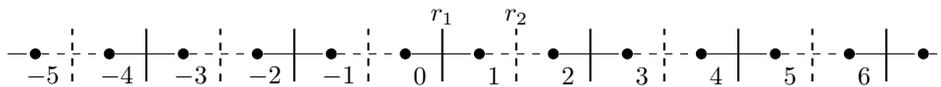
\begin{figure}[htbp]
    \centering
    \resizebox{\textwidth}{!}{

    \begin{tikzpicture}[shorten >=1pt, node distance=2cm, auto]
            \node (1) at (0, 0) {};
            \node (2) at (13, 0) {};
            \node (3) at (1, -0.5) {};
            \node (4) at (1, 0.5) {};
            \node at (0.6, -0.3) {$-5$};
            \node (5) at (2, -0.5) {};
            \node (6) at (2, 0.5) {};
            \node at (1.6, -0.3) {$-4$};
            \node (7) at (3, -0.5) {};
            \node (8) at (3, 0.5) {};
            \node at (2.6, -0.3) {$-3$};
            \node (9) at (4, -0.5) {};
            \node (10) at (4, 0.5) {};
            \node at (3.6, -0.3) {$-2$};
            \node (11) at (5, -0.5) {};
            \node (12) at (5, 0.5) {};
            \node at (4.6, -0.3) {$-1$};
            \node (13) at (6, -0.5) {};
            \node (14) at (6, 0.5) {};
            \node at (5.7, -0.3) {$0$};
            \node (15) at (7, -0.5) {};
            \node (16) at (7, 0.5) {};
            \node at (6.7, -0.3) {$1$};
            \node (17) at (8, -0.5) {};
            \node (18) at (8, 0.5) {};
            \node at (7.7, -0.3) {$2$};
            \node (19) at (9, -0.5) {};
            \node (20) at (9, 0.5) {};
            \node at (8.7, -0.3) {$3$};
            \node (21) at (10, -0.5) {};
            \node (22) at (10, 0.5) {};
            \node at (9.7, -0.3) {$4$};
            \node (23) at (11, -0.5) {};
            \node (24) at (11, 0.5) {};
            \node at (10.7, -0.3) {$5$};
            \node (25) at (12, -0.5) {};
            \node (26) at (12, 0.5) {};
            \node at (11.7, -0.3) {$6$};
            \node at (6, 0.5) {$r_1$};
            \node at (7, 0.5) {$r_2$};
            \node (27) at (0.5, 0) [circle, fill=black, inner sep=1.5pt] {};
            \node (28) at (1.5, 0) [circle, fill=black, inner sep=1.5pt] {};
            \node (29) at (2.5, 0) [circle, fill=black, inner sep=1.5pt] {};
            \node (30) at (3.5, 0) [circle, fill=black, inner sep=1.5pt] {};
            \node (31) at (4.5, 0) [circle, fill=black, inner sep=1.5pt] {};
            \node (32) at (5.5, 0) [circle, fill=black, inner sep=1.5pt] {};
            \node (33) at (6.5, 0) [circle, fill=black, inner sep=1.5pt] {};
            \node (34) at (7.5, 0) [circle, fill=black, inner sep=1.5pt] {};
            \node (35) at (8.5, 0) [circle, fill=black, inner sep=1.5pt] {};
            \node (36) at (9.5, 0) [circle, fill=black, inner sep=1.5pt] {};
            \node (37) at (10.5, 0) [circle, fill=black, inner sep=1.5pt] {};
            \node (38) at (11.5, 0) [circle, fill=black, inner sep=1.5pt] {};
            \node (39) at (12.5, 0) [circle, fill=black, inner sep=1.5pt] {};
            
            \draw[thick, dashed] (3) -- (4);
            \draw[thick] (5) -- (6);
            \draw[thick, dashed] (7) -- (8);
            \draw[thick] (9) -- (10);
            \draw[thick, dashed] (11) -- (12);
            \draw[thick] (13) -- (14);
            \draw[thick, dashed] (15) -- (16);
            \draw[thick] (17) -- (18);
            \draw[thick, dashed] (19) -- (20);
            \draw[thick] (21) -- (22);
            \draw[thick, dashed] (23) -- (24);
            \draw[thick] (25) -- (26);
            \draw (1) -- (27);
            \draw (28) -- (29);
            \draw (30) -- (31);
            \draw (32) -- (33);
            \draw (34) -- (35);
            \draw (36) -- (37);
            \draw (38) -- (39);
            \draw[dashed] (27) -- (28);
            \draw[dashed] (29) -- (30);
            \draw[dashed] (31) -- (32);
            \draw[dashed] (33) -- (34);
            \draw[dashed] (35) -- (36);
            \draw[dashed] (37) -- (38);
            \draw[dashed] (39) -- (2);
            
        \end{tikzpicture}}
    \caption{The Cayley graph of $I_2(\infty)$}
    \label{fig:Cayley graph infinite dihedral group}
    \end{figure}

\end{example}

For the $\widetilde{A}_1$ case (Fig.~\ref{fig:Cayley graph infinite dihedral group}), setting $u = -a \leq 0$, $v = 0$, and $w = b \geq 0$ reduces the hull metric inequality to $(1+a)(1+b) \geq a+b+1$. This holds for all $a, b \geq 0$, verifying Conjecture~\ref{strong hull conj}.

\begin{theorem}
    The Cayley graph for the type $\widetilde{A}_1$ has the strong hull property.
\end{theorem}

Since the elements of $W$ are in bijection with the chambers, the Cayley graph $\mathrm{Cay}(W,S)$ can be geometrically realized as the \textit{dual graph} of the Coxeter complex. By placing a vertex at the barycenter of each chamber and connecting adjacent chambers through their shared panels, the Cayley graph for type $\widetilde{A}_2$ forms a regular hexagonal tiling (illustrated by the dashed lines in Fig.~\ref{fig:Cayley graph affine symmetric group}). Similar dual graph constructions apply to the planar tessellations of $\widetilde{C}_2$ and $\widetilde{G}_2$.

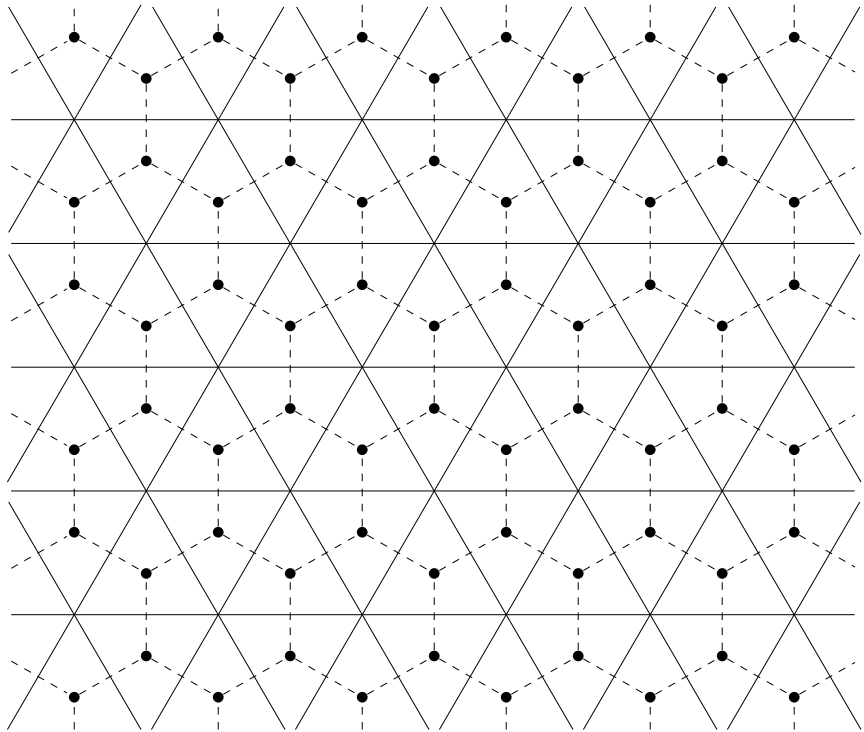
\begin{figure}[htbp]
\centering
\resizebox{\textwidth}{!}{

    \begin{tikzpicture}[shorten >=1pt, node distance=2cm, auto]
    \node (1) at (0, 0) {};
    \node (11) at (6, {6*sqrt(3)}) {};
    \node (2) at (2, 0) {};
    \node (12) at (8, {6*sqrt(3)}) {};
    \node (3) at (4, 0) {};
    \node (13) at (10, {6*sqrt(3)}) {};
    \node (4) at (6, 0) {};
    \node (14) at (12, {6*sqrt(3)}) {};
    \node (5) at (8, 0) {};
    \node (23) at (12, {4*sqrt(3)}) {};
    \node (6) at (10, 0) {};
    \node (21) at (12, {2*sqrt(3)}) {};
    \node (7) at (12, 0) {};
    \node (8) at (0, {6*sqrt(3)}) {};
    \node (9) at (2, {6*sqrt(3)}) {};
    \node (10) at (4, {6*sqrt(3)}) {};
    \node (20) at (12, {sqrt(3)}) {};
    \node (22) at (12, {3*sqrt(3)}) {};
    \node (24) at (12, {5*sqrt(3)}) {};
    \node (15) at (0, {sqrt(3)}) {};
    \node (16) at (0, {2*sqrt(3)}) {};
    \node (17) at (0, {3*sqrt(3)}) {};
    \node (18) at (0, {4*sqrt(3)}) {};
    \node (19) at (0, {5*sqrt(3)}) {};
    \node (25) at (1, 0) {};
    \node (26) at (3, 0) {};
    \node (27) at (5, 0) {};
    \node (28) at (7, 0) {};
    \node (29) at (9, 0) {};
    \node (30) at (11, 0) {};
    \node (31) at (0, {2*sqrt(3)/3}) {};
    \node (32) at (1, {sqrt(3)/3}) [circle, fill=black, inner sep=1.5pt] {};
    \node (33) at (2, {2*sqrt(3)/3}) [circle, fill=black, inner sep=1.5pt] {};
    \node (34) at (3, {sqrt(3)/3}) [circle, fill=black, inner sep=1.5pt] {};
    \node (35) at (4, {2*sqrt(3)/3}) [circle, fill=black, inner sep=1.5pt] {};
    \node (36) at (5, {sqrt(3)/3}) [circle, fill=black, inner sep=1.5pt] {};
    \node (37) at (6, {2*sqrt(3)/3}) [circle, fill=black, inner sep=1.5pt] {};
    \node (38) at (7, {sqrt(3)/3}) [circle, fill=black, inner sep=1.5pt] {};
    \node (39) at (8, {2*sqrt(3)/3}) [circle, fill=black, inner sep=1.5pt] {};
    \node (40) at (9, {sqrt(3)/3}) [circle, fill=black, inner sep=1.5pt] {};
    \node (41) at (10, {2*sqrt(3)/3}) [circle, fill=black, inner sep=1.5pt] {};
    \node (42) at (11, {sqrt(3)/3}) [circle, fill=black, inner sep=1.5pt] {};
    \node (43) at (12, {2*sqrt(3)/3}) {};
    \node (44) at (0, {4*sqrt(3)/3}) {};
    \node (45) at (1, {5*sqrt(3)/3}) [circle, fill=black, inner sep=1.5pt] {};
    \node (46) at (2, {4*sqrt(3)/3}) [circle, fill=black, inner sep=1.5pt] {};
    \node (47) at (3, {5*sqrt(3)/3}) [circle, fill=black, inner sep=1.5pt] {};
    \node (48) at (4, {4*sqrt(3)/3}) [circle, fill=black, inner sep=1.5pt] {};
    \node (49) at (5, {5*sqrt(3)/3}) [circle, fill=black, inner sep=1.5pt] {};
    \node (50) at (6, {4*sqrt(3)/3}) [circle, fill=black, inner sep=1.5pt] {};
    \node (51) at (7, {5*sqrt(3)/3}) [circle, fill=black, inner sep=1.5pt] {};
    \node (52) at (8, {4*sqrt(3)/3}) [circle, fill=black, inner sep=1.5pt] {};
    \node (53) at (9, {5*sqrt(3)/3}) [circle, fill=black, inner sep=1.5pt] {};
    \node (54) at (10, {4*sqrt(3)/3}) [circle, fill=black, inner sep=1.5pt] {};
    \node (55) at (11, {5*sqrt(3)/3}) [circle, fill=black, inner sep=1.5pt] {};
    \node (56) at (12, {4*sqrt(3)/3}) {};
    \node (57) at (0, {8*sqrt(3)/3}) {};
    \node (58) at (1, {7*sqrt(3)/3}) [circle, fill=black, inner sep=1.5pt] {};
    \node (59) at (2, {8*sqrt(3)/3}) [circle, fill=black, inner sep=1.5pt] {};
    \node (60) at (3, {7*sqrt(3)/3}) [circle, fill=black, inner sep=1.5pt] {};
    \node (61) at (4, {8*sqrt(3)/3}) [circle, fill=black, inner sep=1.5pt] {};
    \node (62) at (5, {7*sqrt(3)/3}) [circle, fill=black, inner sep=1.5pt] {};
    \node (63) at (6, {8*sqrt(3)/3}) [circle, fill=black, inner sep=1.5pt] {};
    \node (64) at (7, {7*sqrt(3)/3}) [circle, fill=black, inner sep=1.5pt] {};
    \node (65) at (8, {8*sqrt(3)/3}) [circle, fill=black, inner sep=1.5pt] {};
    \node (66) at (9, {7*sqrt(3)/3}) [circle, fill=black, inner sep=1.5pt] {};
    \node (67) at (10, {8*sqrt(3)/3}) [circle, fill=black, inner sep=1.5pt] {};
    \node (68) at (11, {7*sqrt(3)/3}) [circle, fill=black, inner sep=1.5pt] {};
    \node (69) at (12, {8*sqrt(3)/3}) {};
    \node (70) at (0, {10*sqrt(3)/3}) {};
    \node (71) at (1, {11*sqrt(3)/3}) [circle, fill=black, inner sep=1.5pt] {};
    \node (72) at (2, {10*sqrt(3)/3}) [circle, fill=black, inner sep=1.5pt] {};
    \node (73) at (3, {11*sqrt(3)/3}) [circle, fill=black, inner sep=1.5pt] {};
    \node (74) at (4, {10*sqrt(3)/3}) [circle, fill=black, inner sep=1.5pt] {};
    \node (75) at (5, {11*sqrt(3)/3}) [circle, fill=black, inner sep=1.5pt] {};
    \node (76) at (6, {10*sqrt(3)/3}) [circle, fill=black, inner sep=1.5pt] {};
    \node (77) at (7, {11*sqrt(3)/3}) [circle, fill=black, inner sep=1.5pt] {};
    \node (78) at (8, {10*sqrt(3)/3}) [circle, fill=black, inner sep=1.5pt] {};
    \node (79) at (9, {11*sqrt(3)/3}) [circle, fill=black, inner sep=1.5pt] {};
    \node (80) at (10, {10*sqrt(3)/3}) [circle, fill=black, inner sep=1.5pt] {};
    \node (81) at (11, {11*sqrt(3)/3}) [circle, fill=black, inner sep=1.5pt] {};
    \node (82) at (12, {10*sqrt(3)/3}) {};
    \node (83) at (0, {14*sqrt(3)/3}) {};
    \node (84) at (1, {13*sqrt(3)/3}) [circle, fill=black, inner sep=1.5pt] {};
    \node (85) at (2, {14*sqrt(3)/3}) [circle, fill=black, inner sep=1.5pt] {};
    \node (86) at (3, {13*sqrt(3)/3}) [circle, fill=black, inner sep=1.5pt] {};
    \node (87) at (4, {14*sqrt(3)/3}) [circle, fill=black, inner sep=1.5pt] {};
    \node (88) at (5, {13*sqrt(3)/3}) [circle, fill=black, inner sep=1.5pt] {};
    \node (89) at (6, {14*sqrt(3)/3}) [circle, fill=black, inner sep=1.5pt] {};
    \node (90) at (7, {13*sqrt(3)/3}) [circle, fill=black, inner sep=1.5pt] {};
    \node (91) at (8, {14*sqrt(3)/3}) [circle, fill=black, inner sep=1.5pt] {};
    \node (92) at (9, {13*sqrt(3)/3}) [circle, fill=black, inner sep=1.5pt] {};
    \node (93) at (10, {14*sqrt(3)/3}) [circle, fill=black, inner sep=1.5pt] {};
    \node (94) at (11, {13*sqrt(3)/3}) [circle, fill=black, inner sep=1.5pt] {};
    \node (95) at (12, {14*sqrt(3)/3}) {};
    \node (96) at (0, {16*sqrt(3)/3}) {};
    \node (97) at (1, {17*sqrt(3)/3}) [circle, fill=black, inner sep=1.5pt] {};
    \node (98) at (2, {16*sqrt(3)/3}) [circle, fill=black, inner sep=1.5pt] {};
    \node (99) at (3, {17*sqrt(3)/3}) [circle, fill=black, inner sep=1.5pt] {};
    \node (100) at (4, {16*sqrt(3)/3}) [circle, fill=black, inner sep=1.5pt] {};
    \node (101) at (5, {17*sqrt(3)/3}) [circle, fill=black, inner sep=1.5pt] {};
    \node (102) at (6, {16*sqrt(3)/3}) [circle, fill=black, inner sep=1.5pt] {};
    \node (103) at (7, {17*sqrt(3)/3}) [circle, fill=black, inner sep=1.5pt] {};
    \node (104) at (8, {16*sqrt(3)/3}) [circle, fill=black, inner sep=1.5pt] {};
    \node (105) at (9, {17*sqrt(3)/3}) [circle, fill=black, inner sep=1.5pt] {};
    \node (106) at (10, {16*sqrt(3)/3}) [circle, fill=black, inner sep=1.5pt] {};
    \node (107) at (11, {17*sqrt(3)/3}) [circle, fill=black, inner sep=1.5pt] {};
    \node (108) at (12, {16*sqrt(3)/3}) {};
    \node (109) at (1, {6*sqrt(3)})  {};
    \node (110) at (3, {6*sqrt(3)}) {};
    \node (111) at (5, {6*sqrt(3)}) {};
    \node (112) at (7, {6*sqrt(3)}) {};
    \node (113) at (9, {6*sqrt(3)}) {};
    \node (114) at (11, {6*sqrt(3)}) {};

    \draw (1) -- (11);
    \draw (2) -- (12);
    \draw (3) -- (13);
    \draw (4) -- (14);
    \draw (5) -- (23);
    \draw (6) -- (21);
    \draw (7) -- (11);
    \draw (4) -- (8);
    \draw (5) -- (9);
    \draw (6) -- (10);
    \draw (3) -- (18);
    \draw (2) -- (16);
    \draw (16) -- (10);
    \draw (9) -- (18);
    \draw (21) -- (12);
    \draw (23) -- (13);
    \draw (15) -- (20);
    \draw (16) -- (21);
    \draw (17) -- (22);
    \draw (18) -- (23);
    \draw (19) -- (24);
    \draw[dashed] (31) -- (32);
    \draw[dashed] (32) -- (33);
    \draw[dashed] (33) -- (34);
    \draw[dashed] (34) -- (35);
    \draw[dashed] (35) -- (36);
    \draw[dashed] (36) -- (37);
    \draw[dashed] (37) -- (38);
    \draw[dashed] (38) -- (39);
    \draw[dashed] (39) -- (40);
    \draw[dashed] (40) -- (41);
    \draw[dashed] (41) -- (42);
    \draw[dashed] (42) -- (43);
    \draw[dashed] (25) -- (32);
    \draw[dashed] (26) -- (34);
    \draw[dashed] (27) -- (36);
    \draw[dashed] (28) -- (38);
    \draw[dashed] (29) -- (40);
    \draw[dashed] (30) -- (42);
    \draw[dashed] (44) -- (45);
    \draw[dashed] (45) -- (46);
    \draw[dashed] (46) -- (47);
    \draw[dashed] (47) -- (48);
    \draw[dashed] (48) -- (49);
    \draw[dashed] (49) -- (50);
    \draw[dashed] (50) -- (51);
    \draw[dashed] (51) -- (52);
    \draw[dashed] (52) -- (53);
    \draw[dashed] (53) -- (54);
    \draw[dashed] (54) -- (55);
    \draw[dashed] (55) -- (56);
    \draw[dashed] (33) -- (46);
    \draw[dashed] (35) -- (48);
    \draw[dashed] (37) -- (50);
    \draw[dashed] (39) -- (52);
    \draw[dashed] (41) -- (54);
    \draw[dashed] (45) -- (58);
    \draw[dashed] (47) -- (60);
    \draw[dashed] (49) -- (62);
    \draw[dashed] (51) -- (64);
    \draw[dashed] (53) -- (66);
    \draw[dashed] (55) -- (68);
    \draw[dashed] (57) -- (58);
    \draw[dashed] (58) -- (59);
    \draw[dashed] (59) -- (60);
    \draw[dashed] (60) -- (61);
    \draw[dashed] (61) -- (62);
    \draw[dashed] (62) -- (63);
    \draw[dashed] (63) -- (64);
    \draw[dashed] (64) -- (65);
    \draw[dashed] (65) -- (66);
    \draw[dashed] (66) -- (67);
    \draw[dashed] (67) -- (68);
    \draw[dashed] (68) -- (69);
    \draw[dashed] (59) -- (72);
    \draw[dashed] (61) -- (74);
    \draw[dashed] (63) -- (76);
    \draw[dashed] (65) -- (78);
    \draw[dashed] (67) -- (80);
    \draw[dashed] (70) -- (71);
    \draw[dashed] (71) -- (72);
    \draw[dashed] (72) -- (73);
    \draw[dashed] (73) -- (74);
    \draw[dashed] (74) -- (75);
    \draw[dashed] (75) -- (76);
    \draw[dashed] (76) -- (77);
    \draw[dashed] (77) -- (78);
    \draw[dashed] (78) -- (79);
    \draw[dashed] (79) -- (80);
    \draw[dashed] (80) -- (81);
    \draw[dashed] (81) -- (82);
    \draw[dashed] (83) -- (84);
    \draw[dashed] (84) -- (85);
    \draw[dashed] (85) -- (86);
    \draw[dashed] (86) -- (87);
    \draw[dashed] (87) -- (88);
    \draw[dashed] (88) -- (89);
    \draw[dashed] (89) -- (90);
    \draw[dashed] (90) -- (91);
    \draw[dashed] (91) -- (92);
    \draw[dashed] (92) -- (93);
    \draw[dashed] (93) -- (94);
    \draw[dashed] (94) -- (95);
    \draw[dashed] (71) -- (84);
    \draw[dashed] (73) -- (86);
    \draw[dashed] (75) -- (88);
    \draw[dashed] (77) -- (90);
    \draw[dashed] (79) -- (92);
    \draw[dashed] (81) -- (94);
    \draw[dashed] (98) -- (85);
    \draw[dashed] (100) -- (87);
    \draw[dashed] (102) -- (89);
    \draw[dashed] (104) -- (91);
    \draw[dashed] (106) -- (93);
    \draw[dashed] (97) -- (109);
    \draw[dashed] (99) -- (110);
    \draw[dashed] (111) -- (101);
    \draw[dashed] (112) -- (103);
    \draw[dashed] (113) -- (105);
    \draw[dashed] (114) -- (107);
    \draw[dashed] (96) -- (97);
    \draw[dashed] (97) -- (98);
    \draw[dashed] (98) -- (99);
    \draw[dashed] (99) -- (100);
    \draw[dashed] (100) -- (101);
    \draw[dashed] (101) -- (102);
    \draw[dashed] (102) -- (103);
    \draw[dashed] (103) -- (104);
    \draw[dashed] (104) -- (105);
    \draw[dashed] (105) -- (106);
    \draw[dashed] (106) -- (107);
    \draw[dashed] (107) -- (108);
            
        \end{tikzpicture}}
    \caption{The Cayley graph for $\widetilde{A}_2$}
    \label{fig:Cayley graph affine symmetric group}
\end{figure}

\section{Strong hull property for rank-3 affine irreducible cases}\label{proof of main}

\subsection{Affine type \texorpdfstring{$\widetilde{A}_2$}{A2}}\label{affine type A2}

In order to verify the strong hull inequality \eqref{strong hull property} for type $\widetilde{A}_2$, we will consider its geometric interpretation, refer to Section~4.3 of Humphreys \cite{humphreys1990reflection}. In fact, we have discussed it in Example~\ref{Coxeter complexes as buildings}.

As established through Lemma~\ref{cannot cross wall twice} and Proposition~\ref{union is convex hull}, the convex hull of elements $u$ and $v$ constitutes the union of all the minimal galleries that connect them. This configuration simultaneously forms the maximal gallery structure derived from wall constraints. Thus, all convex hulls in $\mathrm{Cay}(\widetilde{A}_2)$ are hexagons or degenerate hexagons, since they are connected polygons without angles of 240° or 300°. A concrete instantiation of this geometric principle appears in Fig.~\ref{Convex hull of $u$ and $v$ is shaded}.

    \begin{figure}[htbp]
    \centering

    \begin{tikzpicture}[shorten >=1pt, node distance=2cm, auto]
            
            \fill[black!20, opacity=1] (0.5, {sqrt(3)/2}) -- (1.5, {sqrt(3)/2}) -- (1, {sqrt(3)}) -- cycle;
            \fill[black!20, opacity=1] (4, {2*sqrt(3)}) -- (5, {2*sqrt(3)}) -- (4.5, {5*sqrt(3)/2}) -- cycle;
            \fill[black!20, opacity=0.3] (0.5, {sqrt(3)/2}) -- (3.5, {sqrt(3)/2}) -- (5, {2*sqrt(3)}) -- (4.5, {5*sqrt(3)/2}) -- (2.5, {5*sqrt(3)/2}) -- cycle;
            
            \node (1) at (0, 0) {};
            \node (11) at (3, {3*sqrt(3)}) {};
            \node (2) at (1, 0) {};
            \node (12) at (4, {3*sqrt(3)}) {};
            \node (3) at (2, 0) {};
            \node (13) at (5, {3*sqrt(3)}) {};
            \node (4) at (3, 0) {};
            \node (14) at (6, {3*sqrt(3)}) {};
            \node (5) at (4, 0) {};
            \node (23) at (6, {2*sqrt(3)}) {};
            \node (6) at (5, 0) {};
            \node (21) at (6, {sqrt(3)}) {};
            \node (7) at (6, 0) {};
            \node (8) at (0, {3*sqrt(3)}) {};
            \node (9) at (1, {3*sqrt(3)}) {};
            \node (10) at (2, {3*sqrt(3)}) {};
            \node (20) at (6, {0.5*sqrt(3)}) {};
            \node (22) at (6, {1.5*sqrt(3)}) {};
            \node (24) at (6, {2.5*sqrt(3)}) {};
            \node (15) at (0, {0.5*sqrt(3)}) {};
            \node (16) at (0, {sqrt(3)}) {};
            \node (17) at (0, {1.5*sqrt(3)}) {};
            \node (18) at (0, {2*sqrt(3)}) {};
            \node (19) at (0, {2.5*sqrt(3)}) {};
            \node (32) at (1, {2*sqrt(3)/3}) {$u$};
            \node (33) at (4.5, {6.5*sqrt(3)/3}) {$v$};

            \draw (1) -- (11);
            \draw (2) -- (12);
            \draw (3) -- (13);
            \draw (4) -- (14);
            \draw (5) -- (23);
            \draw (6) -- (21);
            \draw (7) -- (11);
            \draw (4) -- (8);
            \draw (5) -- (9);
            \draw (6) -- (10);
            \draw (3) -- (18);
            \draw (2) -- (16);
            \draw (16) -- (10);
            \draw (9) -- (18);
            \draw (21) -- (12);
            \draw (23) -- (13);
            \draw (15) -- (20);
            \draw (16) -- (21);
            \draw (17) -- (22);
            \draw (18) -- (23);
            \draw (19) -- (24);

        \end{tikzpicture}%
    \caption{Convex hull of $u$ and $v$ is shaded.}
    \label{Convex hull of $u$ and $v$ is shaded}
    \end{figure}
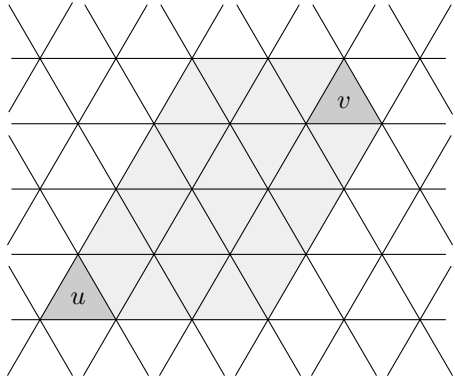

In order to verify inequality \eqref{strong hull property} for $\widetilde{A}_2$, we examine the positions of the three elements $u$, $v$, and $w$ in $\mathrm{Conv}(u,v,w)$. They must lie at the vertices of the (possibly degenerate) hexagon; otherwise, a reduction is applied. Thus, three primary cases arise. In the discussion that follows, although our application of \textit{reduction techniques} may introduce additional scenarios, essentially only the three cases illustrated in Fig.~\ref{fig:three main cases for A2} are relevant.

\begin{figure}[htbp]
    \centering
    \noindent
    %

    \begin{minipage}[b]{0.31\textwidth}
        \centering
        \subfloat[Case~1 for $\widetilde{A}_2$]{%
        \resizebox{\linewidth}{!}{%
            \begin{tikzpicture}[shorten >=1pt, baseline=(current bounding box.center)]\label{A2 case 1}
                \fill[black!20] (0,0) -- (1,0) -- (0.5,{0.5*sqrt(3)}) -- cycle;
                \fill[black!20] (-1,{3*sqrt(3)}) -- (0,{3*sqrt(3)}) -- (-0.5,{3.5*sqrt(3)}) -- cycle;
                \fill[black!20] (7,{2*sqrt(3)}) -- (6.5,{2.5*sqrt(3)}) -- (7.5,{2.5*sqrt(3)}) -- cycle;
                \draw (0,0) -- (1,0) -- (0.5,{0.5*sqrt(3)}) -- cycle;
                \draw (-1,{3*sqrt(3)}) -- (0,{3*sqrt(3)}) -- (-0.5,{3.5*sqrt(3)}) -- cycle;
                \draw (7,{2*sqrt(3)}) -- (6.5,{2.5*sqrt(3)}) -- (7.5,{2.5*sqrt(3)}) -- cycle;
                \draw (0,0) -- (5,0) -- (7.5,{2.5*sqrt(3)}) -- (6.5,{3.5*sqrt(3)}) -- (-0.5,{3.5*sqrt(3)}) -- (-2,{2*sqrt(3)}) -- cycle;
                \node at (0.5, {sqrt(3)/6}) {$u$};
                \node at (-0.5, {3*sqrt(3)+sqrt(3)/6}) {$v$};
                \node at (7, {2.5*sqrt(3)-sqrt(3)/6}) {$w$};
            \end{tikzpicture}%
        }
        }
    \end{minipage}
    \hfill 
    \begin{minipage}[b]{0.31\textwidth}
        \centering
        \subfloat[Case~2 for $\widetilde{A}_2$]{%
        \resizebox{\linewidth}{!}{%
            \begin{tikzpicture}[shorten >=1pt, baseline=(current bounding box.center)]\label{A2 case 2}
                \fill[black!20] (2,0) -- (3,0) -- (2.5,{0.5*sqrt(3)}) -- cycle;
                \fill[black!20] (-2, {2*sqrt(3)}) -- (-1, {2*sqrt(3)}) -- (-1.5, {1.5*sqrt(3)}) -- cycle;
                \fill[black!20] (7, {3*sqrt(3)}) -- (6.5, {3.5*sqrt(3)}) -- (7.5, {3.5*sqrt(3)}) -- cycle;
                \draw (2,0) -- (3,0) -- (2.5,{0.5*sqrt(3)}) -- cycle;
                \draw (-2, {2*sqrt(3)}) -- (-1, {2*sqrt(3)}) -- (-1.5, {1.5*sqrt(3)}) -- cycle;
                \draw (7, {3*sqrt(3)}) -- (6.5, {3.5*sqrt(3)}) -- (7.5, {3.5*sqrt(3)}) -- cycle;
                \draw (0,0) -- (4,0) -- (7.5,{3.5*sqrt(3)}) -- (-0.5,{3.5*sqrt(3)}) -- (-2,{2*sqrt(3)}) -- cycle;
                \node at (2.5, {sqrt(3)/6}) {$u$};
                \node at (-1.5, {2*sqrt(3)-sqrt(3)/6}) {$v$};
                \node at (7, {3.5*sqrt(3)-sqrt(3)/6}) {$w$};
            \end{tikzpicture}%
        }%
        }
    \end{minipage}
    \hfill 
    \begin{minipage}[b]{0.31\textwidth}
        \centering
        \subfloat[Case~3 for $\widetilde{A}_2$]{%
        \resizebox{\linewidth}{!}{%
            \begin{tikzpicture}[shorten >=1pt, baseline=(current bounding box.center)]\label{A2 case 3}
                \fill[black!20] (-2, 0) -- (-1, 0) -- (-1.5, {0.5*sqrt(3)}) -- cycle;
                \fill[black!20] (2, {2*sqrt(3)}) -- (2.5, {1.5*sqrt(3)}) -- (1.5, {1.5*sqrt(3)}) -- cycle;
                \fill[black!20] (7, {3*sqrt(3)}) -- (6.5, {3.5*sqrt(3)}) -- (7.5, {3.5*sqrt(3)}) -- cycle;
                \draw (-2, 0) -- (-1, 0) -- (-1.5, {0.5*sqrt(3)}) -- cycle;
                \draw (2, {2*sqrt(3)}) -- (2.5, {1.5*sqrt(3)}) -- (1.5, {1.5*sqrt(3)}) -- cycle;
                \draw (7, {3*sqrt(3)}) -- (6.5, {3.5*sqrt(3)}) -- (7.5, {3.5*sqrt(3)}) -- cycle;
                \draw (-2,0) -- (4,0) -- (7.5,{3.5*sqrt(3)}) -- (1.5,{3.5*sqrt(3)}) -- cycle;
                \node at (-1.5, {sqrt(3)/6}) {$u$};
                \node at (2, {1.5*sqrt(3)+sqrt(3)/6}) {$v$};
                \node at (7, {3.5*sqrt(3)-sqrt(3)/6}) {$w$};
            \end{tikzpicture}%
        }%
        }
    \end{minipage}

    \caption{Three main cases for $\widetilde{A}_2$}
    \label{fig:three main cases for A2}
\end{figure}
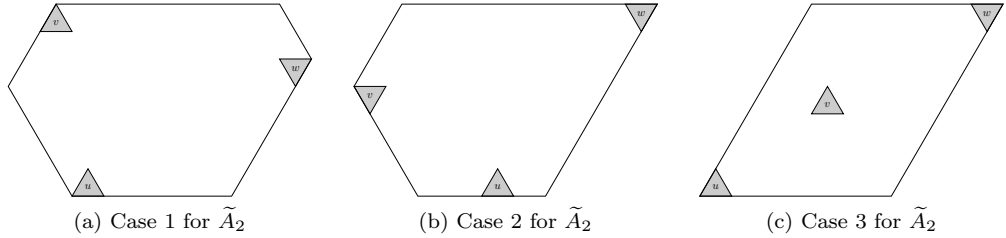

\subsubsection*{Case~1}

Consider the convex hulls of $u$ and $v$, and of $v$ and $w$ in Fig.~\ref{A2 case 1}. These geometric structures are explicitly illustrated in Fig.~\ref{Conv(u,v) & Conv(v,w)}.

    \begin{figure}[htbp]
    \centering

    \begin{tikzpicture}[shorten >=1pt, node distance=2cm, auto]
            
            \fill[black!20] (0,0) -- (1,0) -- (0.5,{0.5*sqrt(3)}) -- cycle;
            \fill[black!20] (-1,{3*sqrt(3)}) -- (0,{3*sqrt(3)}) -- (-0.5,{3.5*sqrt(3)}) -- cycle;
            \fill[black!20] (7,{2*sqrt(3)}) -- (6.5,{2.5*sqrt(3)}) -- (7.5,{2.5*sqrt(3)}) -- cycle;
            \fill[black!20, opacity=0.3] (-1,{3*sqrt(3)}) -- (-0.5,{3.5*sqrt(3)}) -- (6.5,{3.5*sqrt(3)}) -- (7.5,{2.5*sqrt(3)}) -- (7,{2*sqrt(3)}) -- (0,{2*sqrt(3)}) -- cycle;
            \fill[black!20, opacity=0.3] (0,0) -- (1,0) -- (2,{sqrt(3)}) -- (-0.5,{3.5*sqrt(3)}) -- (-2,{2*sqrt(3)}) -- cycle;
            \draw (0,0) -- (1,0) -- (0.5,{0.5*sqrt(3)}) -- cycle;
            \draw (-1,{3*sqrt(3)}) -- (0,{3*sqrt(3)}) -- (-0.5,{3.5*sqrt(3)}) -- cycle;
            \draw (7,{2*sqrt(3)}) -- (6.5,{2.5*sqrt(3)}) -- (7.5,{2.5*sqrt(3)}) -- cycle;
            \draw (0,0) -- (5,0) -- (7.5,{2.5*sqrt(3)}) -- (6.5,{3.5*sqrt(3)}) -- (-0.5,{3.5*sqrt(3)}) -- (-2,{2*sqrt(3)}) -- cycle;
            \node at (0.5, {sqrt(3)/6}) {$u$};
            \node at (-0.5, {3*sqrt(3)+sqrt(3)/6}) {$v$};
            \node at (7, {2.5*sqrt(3)-sqrt(3)/6}) {$w$};

        \end{tikzpicture}%
    \caption{$\mathrm{Conv}(u,v)$ and $\mathrm{Conv}(v,w)$}
    \label{Conv(u,v) & Conv(v,w)}
    \end{figure}

To establish the strong hull inequality \eqref{strong hull property}, we initiate the proof by performing a leftward translation of element $u$ in Fig.~\ref{Conv(u,v) & Conv(v,w)} to position $u'$. This geometric manipulation results in the modified configuration presented in Fig.~\ref{u->u1}.

    \begin{figure}[htbp]
    \centering

    \begin{tikzpicture}[shorten >=1pt, node distance=2cm, auto]
            
            \fill[black!20] (-4,0) -- (-3,0) -- (-3.5,{0.5*sqrt(3)}) -- cycle;
            \fill[black!20] (-1,{3*sqrt(3)}) -- (0,{3*sqrt(3)}) -- (-0.5,{3.5*sqrt(3)}) -- cycle;
            \fill[black!20] (7,{2*sqrt(3)}) -- (6.5,{2.5*sqrt(3)}) -- (7.5,{2.5*sqrt(3)}) -- cycle;
            \fill[black!20, opacity=0.3] (-1,{3*sqrt(3)}) -- (-0.5,{3.5*sqrt(3)}) -- (6.5,{3.5*sqrt(3)}) -- (7.5,{2.5*sqrt(3)}) -- (7,{2*sqrt(3)}) -- (0,{2*sqrt(3)}) -- cycle;
            \fill[black!20, opacity=0.3] (-4,0) -- (-3,0) -- (0,{3*sqrt(3)}) -- (-0.5,{3.5*sqrt(3)}) -- cycle;
            \draw (-4,0) -- (-3,0) -- (-3.5,{0.5*sqrt(3)}) -- cycle;
            \draw (-1,{3*sqrt(3)}) -- (0,{3*sqrt(3)}) -- (-0.5,{3.5*sqrt(3)}) -- cycle;
            \draw (7,{2*sqrt(3)}) -- (6.5,{2.5*sqrt(3)}) -- (7.5,{2.5*sqrt(3)}) -- cycle;
            \draw (-4,0) -- (5,0) -- (7.5,{2.5*sqrt(3)}) -- (6.5,{3.5*sqrt(3)}) -- (-0.5,{3.5*sqrt(3)}) -- cycle;
            \node at (-3.5, {sqrt(3)/6}) {$u_1$};
            \node at (-0.5, {3*sqrt(3)+sqrt(3)/6}) {$v$};
            \node at (7, {2.5*sqrt(3)-sqrt(3)/6}) {$w$};

        \end{tikzpicture}%
    \caption{Translate $u$ to $u_1$}
    \label{u->u1}
    \end{figure}

Through the aforementioned translation procedure, we obtain the inequality
\begin{equation}\label{translation inequality}
    |\mathrm{Conv}(u,v)| \geq d(u,v) + 1 = d(u_1,v) + 1 = |\mathrm{Conv}(u_1,v)|.
\end{equation}
Furthermore, a straightforward verification yields
\begin{equation}\label{monotonicity relation}
    |\mathrm{Conv}(u,v,w)| \geq |\mathrm{Conv}(u_1,v,w)|.
\end{equation}

A direct combination of inequalities \eqref{translation inequality} and \eqref{monotonicity relation} yields the implication
\begin{equation}\label{ineq u'}
    |\mathrm{Conv}(u_1,v)| \cdot |\mathrm{Conv}(v,w)| \geq |\mathrm{Conv}(u_1,v,w)|
\end{equation}
which consequently establishes the strong hull inequality \eqref{strong hull property}.

We now perform an upper-right directional translation of point $w$ from Fig.~\ref{u->u1}, resulting in the configuration shown in Fig.~\ref{w->w1} where $w$ attains its new position $w_1$.

    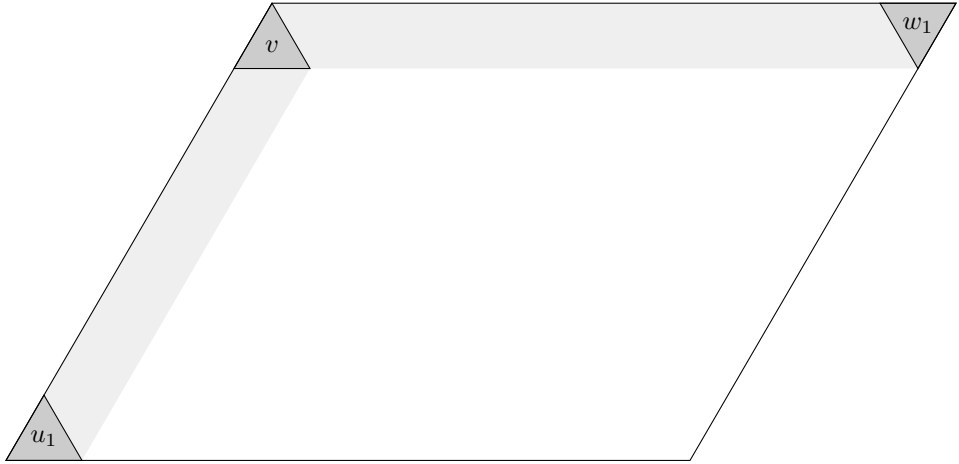
\begin{figure}[htbp]
    \centering

    \begin{tikzpicture}[shorten >=1pt, node distance=2cm, auto]
            
            \fill[black!20] (-4,0) -- (-3,0) -- (-3.5,{0.5*sqrt(3)}) -- cycle;
            \fill[black!20] (-1,{3*sqrt(3)}) -- (0,{3*sqrt(3)}) -- (-0.5,{3.5*sqrt(3)}) -- cycle;
            \fill[black!20] (8,{3*sqrt(3)}) -- (7.5,{3.5*sqrt(3)}) -- (8.5,{3.5*sqrt(3)}) -- cycle;
            \fill[black!20, opacity=0.3] (-1,{3*sqrt(3)}) -- (-0.5,{3.5*sqrt(3)}) -- (8.5,{3.5*sqrt(3)}) -- (8,{3*sqrt(3)}) -- cycle;
            \fill[black!20, opacity=0.3] (-4,0) -- (-3,0) -- (0,{3*sqrt(3)}) -- (-0.5,{3.5*sqrt(3)}) -- cycle;
            \draw (-4,0) -- (-3,0) -- (-3.5,{0.5*sqrt(3)}) -- cycle;
            \draw (-1,{3*sqrt(3)}) -- (0,{3*sqrt(3)}) -- (-0.5,{3.5*sqrt(3)}) -- cycle;
            \draw (8,{3*sqrt(3)}) -- (7.5,{3.5*sqrt(3)}) -- (8.5,{3.5*sqrt(3)}) -- cycle;
            \draw (-4,0) -- (5,0) -- (8.5,{3.5*sqrt(3)}) --  (-0.5,{3.5*sqrt(3)}) -- cycle;
            \node at (-3.5, {sqrt(3)/6}) {$u_1$};
            \node at (-0.5, {3*sqrt(3)+sqrt(3)/6}) {$v$};
            \node at (8, {3.5*sqrt(3)-sqrt(3)/6}) {$w_1$};

        \end{tikzpicture}%
    \caption{Translate $w$ to $w_1$}
    \label{w->w1}
    \end{figure}

By analogous reasoning, we derive the inequalities
\begin{equation}
    |\mathrm{Conv}(v,w)| \geq |\mathrm{Conv}(v,w_1)| \quad \text{and} \quad |\mathrm{Conv}(u_1,v,w)| \leq |\mathrm{Conv}(u_1,v,w_1)|.
\end{equation}
This deduction yields the implication
\begin{equation}\label{implication w'}
    |\mathrm{Conv}(u_1,v)| \cdot |\mathrm{Conv}(v,w_1)| \geq |\mathrm{Conv}(u_1,v,w_1)|
\end{equation}
which consequently establishes inequality \eqref{ineq u'}.

We observe that the key inequality
\begin{equation}\label{key inequality}
    \left(|\mathrm{Conv}(u_1,v)| - 1\right) \cdot \left(|\mathrm{Conv}(v,w_1)| - 1\right) \geq |\mathrm{Conv}(u_1,v,w_1)| - 2
\end{equation}
induces the refined estimate
\begin{equation}
    |\mathrm{Conv}(u_1,v)| \cdot \left(|\mathrm{Conv}(v,w_1)| - 1\right) \geq |\mathrm{Conv}(u_1,v,w_1)| - 1,
\end{equation}
which leads to the ultimate form through successive approximation
\begin{equation}
    |\mathrm{Conv}(u_1,v)| \cdot |\mathrm{Conv}(v,w_1)| \geq |\mathrm{Conv}(u_1,v,w_1)|.
\end{equation}
Consequently, the configuration in Fig.~\ref{w->w1} can be reduced to the essential structure shown in Fig.~\ref{reduction of case 1 A2}.

    \begin{figure}[htbp]
    \centering

    \begin{tikzpicture}[shorten >=1pt, node distance=2cm, auto]
            
            \fill[black!20] (-3,0) -- (-3.5,{0.5*sqrt(3)}) -- (-2.5,{0.5*sqrt(3)}) -- cycle;
            \fill[black!20] (-1,{3*sqrt(3)}) -- (0,{3*sqrt(3)}) -- (-0.5,{3.5*sqrt(3)}) -- cycle;
            \fill[black!20] (7,{3*sqrt(3)}) -- (8,{3*sqrt(3)}) -- (7.5,{3.5*sqrt(3)}) -- cycle;
            \fill[black!20, opacity=0.3] (-1,{3*sqrt(3)}) -- (-0.5,{3.5*sqrt(3)}) -- (7.5,{3.5*sqrt(3)}) -- (8,{3*sqrt(3)}) -- cycle;
            \fill[black!20, opacity=0.3] (-3.5,{0.5*sqrt(3)}) -- (-3,0) -- (0,{3*sqrt(3)}) -- (-0.5,{3.5*sqrt(3)}) -- cycle;
            \draw (-3,0) -- (-3.5,{0.5*sqrt(3)}) -- (-2.5,{0.5*sqrt(3)}) -- cycle;
            \draw (-1,{3*sqrt(3)}) -- (0,{3*sqrt(3)}) -- (-0.5,{3.5*sqrt(3)}) -- cycle;
            \draw (7,{3*sqrt(3)}) -- (8,{3*sqrt(3)}) -- (7.5,{3.5*sqrt(3)}) -- cycle;
            \draw (-3.5,{0.5*sqrt(3)}) -- (-3,0) -- (5,0) -- (8,{3*sqrt(3)}) -- (7.5,{3.5*sqrt(3)}) --  (-0.5,{3.5*sqrt(3)}) -- cycle;
            \node at (-3, {0.5*sqrt(3)-sqrt(3)/6}) {$u_2$};
            \node at (-0.5, {3*sqrt(3)+sqrt(3)/6}) {$v$};
            \node at (7.5, {3*sqrt(3)+sqrt(3)/6}) {$w_2$};

        \end{tikzpicture}%
    \caption{Reduction of Case~1 for $\widetilde{A}_2$}
    \label{reduction of case 1 A2}
    \end{figure}
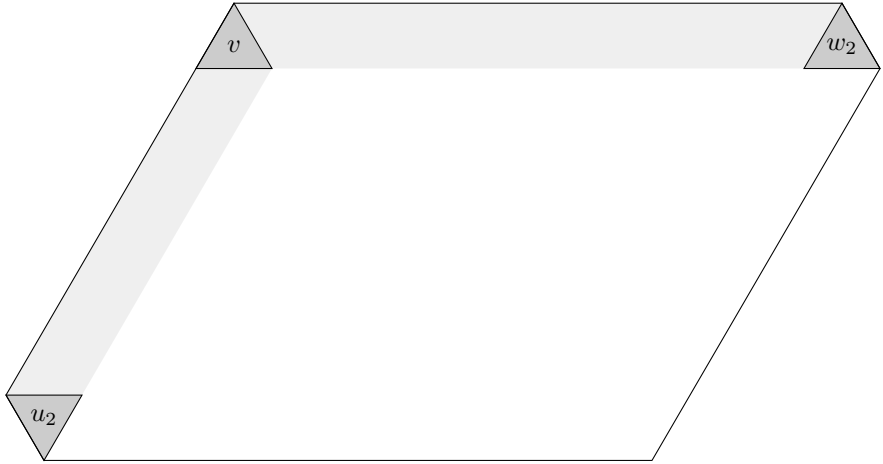

A direct computation reveals the cardinality relations:
\[
\begin{aligned}
|\mathrm{Conv}(u_2,v)| &= |\mathrm{Conv}(u_1,v)| - 1, \\
|\mathrm{Conv}(v,w_2)| &= |\mathrm{Conv}(v,w_1)| - 1, \\
|\mathrm{Conv}(u_2,v,w_2)| &= |\mathrm{Conv}(u_1,v,w_1)| - 2.
\end{aligned}
\]
Consequently, the proof of Case~1 reduces to verifying the fundamental inequality
\begin{equation}\label{base inequality}
|\mathrm{Conv}(u_2,v)| \cdot |\mathrm{Conv}(v,w_2)| \geq |\mathrm{Conv}(u_2,v,w_2)|
\end{equation}
in the reduced configuration depicted in Fig.~\ref{reduction of case 1 A2}.

\subsubsection*{Case~2}

Consider the element $u$ in Fig.~\ref{A2 case 2}, which is potentially situated at the lower-right corner of the convex hull, that is, the (degenerate) hexagon. By performing a leftward translation of $u$ to the position illustrated in Fig.~\ref{u positioning along the bottom wall} below (specifically the lower-left corner), we observe that the following inequality holds throughout this geometric transformation
\begin{equation}
    |\mathrm{Conv}(u_1,v)|= d(u_1,v)+1\leq d(u,v)+1\leq|\mathrm{Conv}(u,v)|.
\end{equation}

    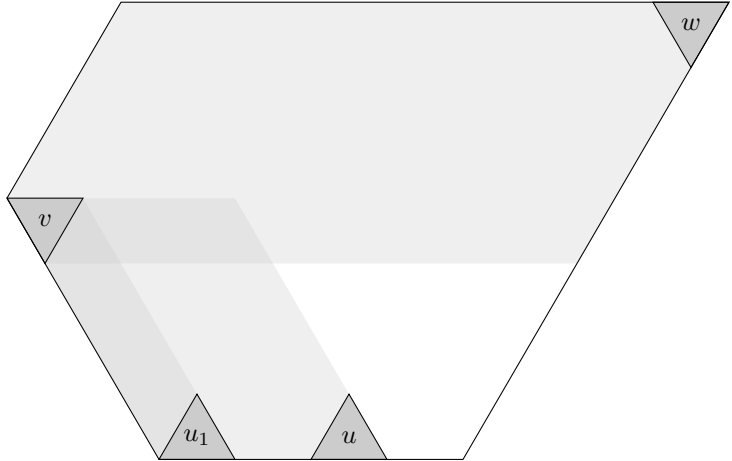
\begin{figure}[htbp]
    \centering

    \begin{tikzpicture}[shorten >=1pt, node distance=2cm, auto]
            
            \fill[black!20] (2,0) -- (3,0) -- (2.5,{0.5*sqrt(3)}) -- cycle;
            \fill[black!20] (-2, {2*sqrt(3)}) -- (-1, {2*sqrt(3)}) -- (-1.5, {1.5*sqrt(3)}) -- cycle;
            \fill[black!20] (7, {3*sqrt(3)}) -- (6.5, {3.5*sqrt(3)}) -- (7.5, {3.5*sqrt(3)}) -- cycle;
            \fill[black!20] (0,0) -- (1,0) -- (0.5,{0.5*sqrt(3)}) -- cycle;
            \fill[black!20, opacity=0.3] (0,0) -- (1,0) -- (-1,{2*sqrt(3)}) -- (-2,{2*sqrt(3)}) -- cycle;
            \fill[black!20, opacity=0.3] (0,0) -- (3,0) -- (1,{2*sqrt(3)}) -- (-2,{2*sqrt(3)}) -- cycle;
            \fill[black!20, opacity=0.3] (-1.5,{1.5*sqrt(3)}) -- (5.5,{1.5*sqrt(3)}) -- (7.5,{3.5*sqrt(3)}) -- (-0.5,{3.5*sqrt(3)}) -- (-2,{2*sqrt(3)}) -- cycle;
            \draw (2,0) -- (3,0) -- (2.5,{0.5*sqrt(3)}) -- cycle;
            \draw (0,0) -- (1,0) -- (0.5,{0.5*sqrt(3)}) -- cycle;
            \draw (-2, {2*sqrt(3)}) -- (-1, {2*sqrt(3)}) -- (-1.5, {1.5*sqrt(3)}) -- cycle;
            \draw (7, {3*sqrt(3)}) -- (6.5, {3.5*sqrt(3)}) -- (7.5, {3.5*sqrt(3)}) -- cycle;
            \draw (0,0) -- (4,0) -- (7.5,{3.5*sqrt(3)}) -- (-0.5,{3.5*sqrt(3)}) -- (-2,{2*sqrt(3)}) -- cycle;
            \node at (0.5, {sqrt(3)/6}) {$u_1$};
            \node at (2.5, {sqrt(3)/6}) {$u$};
            \node at (-1.5, {2*sqrt(3)-sqrt(3)/6}) {$v$};
            \node at (7, {3.5*sqrt(3)-sqrt(3)/6}) {$w$};

        \end{tikzpicture}%
    \caption{$u$ positioning along the bottom wall}
    \label{u positioning along the bottom wall}
    \end{figure}

To establish the inequality
\begin{equation}\label{ineq u case 2 A2}
    |\mathrm{Conv}(u,v)| \cdot |\mathrm{Conv}(v,w)| \geq |\mathrm{Conv}(u,v,w)|,
\end{equation}
it suffices to verify the inequality
\begin{equation}\label{ineq u1 case 2 A2}
    |\mathrm{Conv}(u_1,v)| \cdot |\mathrm{Conv}(v,w)| \geq |\mathrm{Conv}(u_1,v,w)|.
\end{equation}

We perform a leftward translation of element $u$ in Fig.~\ref{A2 case 2}, until it reaches the position $u_2$ as depicted in Fig.~\ref{case2 u2 A2}.

    \begin{figure}[htbp]
    \centering

    \begin{tikzpicture}[shorten >=1pt, node distance=2cm, auto]
            
            \fill[black!20] (-2, {2*sqrt(3)}) -- (-1, {2*sqrt(3)}) -- (-1.5, {1.5*sqrt(3)}) -- cycle;
            \fill[black!20] (7, {3*sqrt(3)}) -- (6.5, {3.5*sqrt(3)}) -- (7.5, {3.5*sqrt(3)}) -- cycle;
            \fill[black!20] (-4,0) -- (-3,0) -- (-3.5,{0.5*sqrt(3)}) -- cycle;
            \fill[black!20, opacity=0.3] (-1.5,{1.5*sqrt(3)}) -- (5.5,{1.5*sqrt(3)}) -- (7.5,{3.5*sqrt(3)}) -- (-0.5,{3.5*sqrt(3)}) -- (-2,{2*sqrt(3)}) -- cycle;
            \fill[black!20, opacity=0.3] (-4,0) -- (-3,0) -- (-1,{2*sqrt(3)}) -- (-2,{2*sqrt(3)}) -- cycle;
            \draw (-4,0) -- (-3,0) -- (-3.5,{0.5*sqrt(3)}) -- cycle;
            \draw (-2, {2*sqrt(3)}) -- (-1, {2*sqrt(3)}) -- (-1.5, {1.5*sqrt(3)}) -- cycle;
            \draw (7, {3*sqrt(3)}) -- (6.5, {3.5*sqrt(3)}) -- (7.5, {3.5*sqrt(3)}) -- cycle;
            \draw (-4,0) -- (4,0) -- (7.5,{3.5*sqrt(3)}) -- (-0.5,{3.5*sqrt(3)}) -- cycle;
            \node at (-3.5, {sqrt(3)/6}) {$u_2$};
            \node at (-1.5, {2*sqrt(3)-sqrt(3)/6}) {$v$};
            \node at (7, {3.5*sqrt(3)-sqrt(3)/6}) {$w$};

        \end{tikzpicture}%
    \caption{Translate $u_1$ to $u_2$}
    \label{case2 u2 A2}
    \end{figure}

Observe that the following equations hold:
\begin{equation}\label{case2 u2v,u,v A2}
    |\mathrm{Conv}(u_2,v)| = d(u_2,v) + 1 = d(u_1,v) + 1 = |\mathrm{Conv}(u_1,v)|,
\end{equation}
and moreover, we have the cardinality inequality
\begin{equation}\label{case2 u2vw A2}
    |\mathrm{Conv}(u_2,v,w)| \geq |\mathrm{Conv}(u_1,v,w)|.
\end{equation}
Consequently, the established inequality
\begin{equation}\label{case2 ineq u2vw A2}
    |\mathrm{Conv}(u_2,v)| \cdot |\mathrm{Conv}(v,w)| \geq |\mathrm{Conv}(u_2,v,w)|
\end{equation}
directly yields the desired inequality \eqref{strong hull property}. To complete the proof, we invoke an analogous argument to that in Case~1, thereby reducing the configuration depicted in Fig.~\ref{case2 u2 A2} to the simplified diagram in Fig.~\ref{Reduction of Case 2 for A2}.

    \begin{figure}[htbp]
    \centering

    \begin{tikzpicture}[shorten >=1pt, node distance=2cm, auto]
            
            \fill[black!20] (-2, {2*sqrt(3)}) -- (-1, {2*sqrt(3)}) -- (-1.5, {1.5*sqrt(3)}) -- cycle;
            \fill[black!20] (6, {3*sqrt(3)}) -- (7, {3*sqrt(3)}) -- (6.5, {3.5*sqrt(3)}) -- cycle;
            \fill[black!20] (-3.5,{0.5*sqrt(3)}) -- (-3,0) -- (-2.5,{0.5*sqrt(3)}) -- cycle;
            \fill[black!20, opacity=0.3] (-1.5,{1.5*sqrt(3)}) -- (5.5,{1.5*sqrt(3)}) -- (7, {3*sqrt(3)}) -- (6.5, {3.5*sqrt(3)}) -- (-0.5,{3.5*sqrt(3)}) -- (-2,{2*sqrt(3)}) -- cycle;
            \fill[black!20, opacity=0.3] (-3.5,{0.5*sqrt(3)}) -- (-3,0) -- (-1,{2*sqrt(3)}) -- (-2,{2*sqrt(3)}) -- cycle;
            \draw (-3.5,{0.5*sqrt(3)}) -- (-3,0) -- (-2.5,{0.5*sqrt(3)}) -- cycle;
            \draw (-2, {2*sqrt(3)}) -- (-1, {2*sqrt(3)}) -- (-1.5, {1.5*sqrt(3)}) -- cycle;
            \draw (6, {3*sqrt(3)}) -- (7, {3*sqrt(3)}) -- (6.5, {3.5*sqrt(3)}) -- cycle;
            \draw (-3.5,{0.5*sqrt(3)}) -- (-3,0) -- (4,0) -- (7, {3*sqrt(3)}) -- (6.5, {3.5*sqrt(3)}) -- (-0.5,{3.5*sqrt(3)}) -- cycle;
            \node at (-3, {0.5*sqrt(3)-sqrt(3)/6}) {$u_3$};
            \node at (-1.5, {2*sqrt(3)-sqrt(3)/6}) {$v$};
            \node at (6.5, {3*sqrt(3)+sqrt(3)/6}) {$w_1$};

        \end{tikzpicture}%
    \caption{Reduction of Case~2 for $\widetilde{A}_2$}
    \label{Reduction of Case 2 for A2}
    \end{figure}
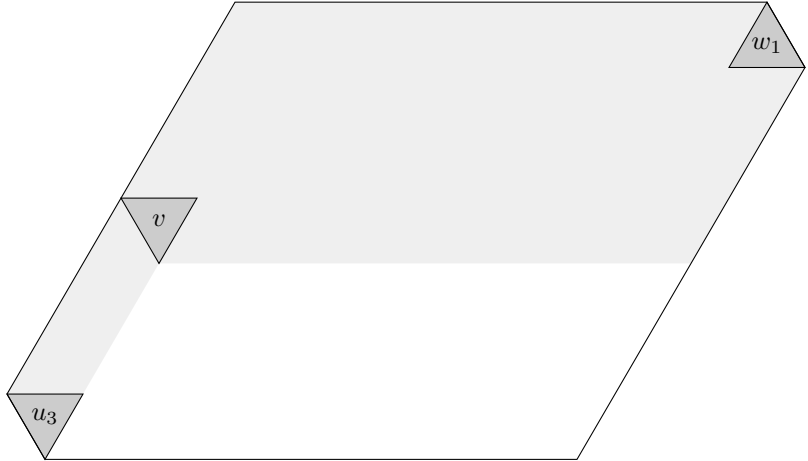

Consequently, it suffices to demonstrate the inequality  
\begin{equation}\label{case2 u3w1 A2}  
    |\mathrm{Conv}(u_3,v)| \cdot |\mathrm{Conv}(v,w_1)| \geq |\mathrm{Conv}(u_3,v,w_1)|,  
\end{equation}  
which corresponds precisely to the framework of Case~1.

\subsubsection*{Case~3}

We must also consider the possibility that $v$ lies within the convex hull $\mathrm{Conv}(u,v)$, specifically when $v$ belongs to the interior of $\mathrm{Conv}(u,v)$ as illustrated in Fig.~\ref{Reduction of Case 3 for A2}. This configuration can be resolved through arguments analogous to those developed in Case~1 and Case~2.

    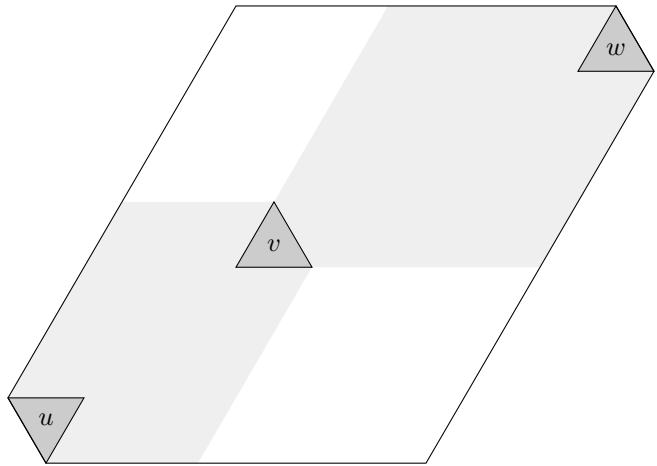
\begin{figure}[htbp]
    \centering

    \begin{tikzpicture}[shorten >=1pt, node distance=2cm, auto]
            
            \fill[black!20] (-1, 0) -- (-1.5, {0.5*sqrt(3)}) -- (-0.5, {0.5*sqrt(3)}) -- cycle;
            \fill[black!20] (2, {2*sqrt(3)}) -- (2.5, {1.5*sqrt(3)}) -- (1.5, {1.5*sqrt(3)}) -- cycle;
            \fill[black!20] (7, {3*sqrt(3)}) -- (6.5, {3.5*sqrt(3)}) -- (6, {3*sqrt(3)}) -- cycle;
            \fill[black!20, opacity=0.3] (-1.5, {0.5*sqrt(3)}) -- (-1, 0) -- (1,0) -- (2.5,{1.5*sqrt(3)}) -- (2,{2*sqrt(3)}) -- (0,{2*sqrt(3)}) -- cycle;
            \fill[black!20, opacity=0.3] (1.5, {1.5*sqrt(3)}) -- (5.5,{1.5*sqrt(3)}) -- (7,{3*sqrt(3)}) -- (6.5,{3.5*sqrt(3)}) -- (3.5,{3.5*sqrt(3)}) -- cycle;
            \draw (-1, 0) -- (-1.5, {0.5*sqrt(3)}) -- (-0.5, {0.5*sqrt(3)}) -- cycle;
            \draw (2, {2*sqrt(3)}) -- (2.5, {1.5*sqrt(3)}) -- (1.5, {1.5*sqrt(3)}) -- cycle;
            \draw (7, {3*sqrt(3)}) -- (6.5, {3.5*sqrt(3)}) -- (6, {3*sqrt(3)}) -- cycle;
            \draw (-1.5, {0.5*sqrt(3)}) -- (-1, 0) -- (4,0) -- (7, {3*sqrt(3)}) -- (6.5, {3.5*sqrt(3)}) -- (1.5,{3.5*sqrt(3)}) -- cycle;
            \node at (-1, {0.5*sqrt(3)-sqrt(3)/6}) {$u$};
            \node at (2, {1.5*sqrt(3)+sqrt(3)/6}) {$v$};
            \node at (6.5, {3*sqrt(3)+sqrt(3)/6}) {$w$};

        \end{tikzpicture}%
    \caption{Reduction of Case~3 for $\widetilde{A}_2$}
    \label{Reduction of Case 3 for A2}
    \end{figure}

\subsubsection*{Formulas and computations}

We have reduced all configurations to three cases involving convex hulls shaped as parallelograms with two truncated corners, as illustrated in Fig.~\ref{reduction of case 1 A2}, Fig.~\ref{Reduction of Case 2 for A2}, and Fig.~\ref{Reduction of Case 3 for A2}. We now proceed to investigate the formulas for strong hulls in these reduced configurations. By taking the chamber $u$ as the origin, we can introduce a Cartesian coordinate system as depicted in Fig.~\ref{Cartesian coordinate}. Within this coordinate framework, each chamber admits a unique coordinate representation. For example, in Fig.~\ref{Cartesian coordinate}, the coordinates of the chambers $u$ and $v$ are explicitly given as $(0,0)$ and $(7,3)$, respectively.

    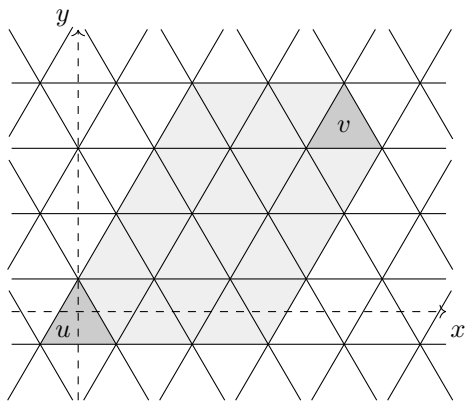
\begin{figure}[htbp]
    \centering

    \begin{tikzpicture}[shorten >=1pt, node distance=2cm, auto]
            
            \fill[black!20, opacity=1] (0.5, {sqrt(3)/2}) -- (1.5, {sqrt(3)/2}) -- (1, {sqrt(3)}) -- cycle;
            \fill[black!20, opacity=1] (4, {2*sqrt(3)}) -- (5, {2*sqrt(3)}) -- (4.5, {5*sqrt(3)/2}) -- cycle;
            \fill[black!20, opacity=0.3] (0.5, {sqrt(3)/2}) -- (3.5, {0.5*sqrt(3)}) -- (5, {2*sqrt(3)}) -- (4.5, {5*sqrt(3)/2}) -- (2.5, {5*sqrt(3)/2}) -- cycle;
            
            \node (1) at (0, 0) {};
            \node (11) at (3, {3*sqrt(3)}) {};
            \node (2) at (1, 0) {};
            \node (12) at (4, {3*sqrt(3)}) {};
            \node (3) at (2, 0) {};
            \node (13) at (5, {3*sqrt(3)}) {};
            \node (4) at (3, 0) {};
            \node (14) at (6, {3*sqrt(3)}) {};
            \node (5) at (4, 0) {};
            \node (23) at (6, {2*sqrt(3)}) {};
            \node (6) at (5, 0) {};
            \node (21) at (6, {sqrt(3)}) {};
            \node (7) at (6, 0) {};
            \node (8) at (0, {3*sqrt(3)}) {};
            \node (9) at (1, {3*sqrt(3)}) {};
            \node (10) at (2, {3*sqrt(3)}) {};
            \node (20) at (6, {0.5*sqrt(3)}) {};
            \node (22) at (6, {1.5*sqrt(3)}) {};
            \node (24) at (6, {2.5*sqrt(3)}) {};
            \node (15) at (0, {0.5*sqrt(3)}) {};
            \node (16) at (0, {sqrt(3)}) {};
            \node (17) at (0, {1.5*sqrt(3)}) {};
            \node (18) at (0, {2*sqrt(3)}) {};
            \node (19) at (0, {2.5*sqrt(3)}) {};
            \node (32) at (0.8, {1.8*sqrt(3)/3}) {$u$};
            \node (33) at (4.5, {6.5*sqrt(3)/3}) {$v$};
            \node (101) at (0, {0.75*sqrt(3)}) {};
            \node (102) at (6, {0.75*sqrt(3)}) {};
            \node (103) at (1, 0) {};
            \node (104) at (1, {3*sqrt(3)}) {};
            \node (105) at (6, {0.6*sqrt(3)}) {$x$};
            \node (106) at (0.8, {3*sqrt(3)}) {$y$};

            \draw (1) -- (11);
            \draw (2) -- (12);
            \draw (3) -- (13);
            \draw (4) -- (14);
            \draw (5) -- (23);
            \draw (6) -- (21);
            \draw (7) -- (11);
            \draw (4) -- (8);
            \draw (5) -- (9);
            \draw (6) -- (10);
            \draw (3) -- (18);
            \draw (2) -- (16);
            \draw (16) -- (10);
            \draw (9) -- (18);
            \draw (21) -- (12);
            \draw (23) -- (13);
            \draw (15) -- (20);
            \draw (16) -- (21);
            \draw (17) -- (22);
            \draw (18) -- (23);
            \draw (19) -- (24);
            \draw[->, dashed] (101) -- (102);
            \draw[->, dashed] (103) -- (104);

        \end{tikzpicture}%
    \caption{Cartesian coordinate in the Cayley graph for $\widetilde{A}_2$}
    \label{Cartesian coordinate}
    \end{figure}

Fig.~\ref{Cartesian coordinate} depicts a special configuration where the chamber $u$ assumes the shape of a downward-pointing triangle ($\bigtriangledown$). More generally, it may also form an upward-pointing triangle ($\bigtriangleup$), as will be discussed subsequently. We first analyze the $\bigtriangleup$ configuration. Let $v$ be positioned at coordinate $(x,y)$, where $x+y$ must be a positive even integer. The convex hull $\mathrm{Conv}(u,v)$, which forms a parallelogram with the top-right chamber removed, contains $y+1$ horizontal rows of chambers. Specifically, the uppermost row consists of $x-y+1$ chambers, while each subsequent row contains $x-y+2$ chambers. Through this geometric decomposition, we deduce the cardinality of chambers in the truncated convex hull:
\begin{equation}\label{coordinateformula2}
    |\mathrm{Conv}(u,v)| = (x-y+1) + y(x-y+2) = xy + x - y^2 + y + 1,
\end{equation}
valid under the parity constraint that $x+y$ is a positive even integer.

When $u$ is configured as a downward-pointing triangle ($\bigtriangledown$), let $v$ be positioned at coordinates $(x,y)$ where the parity condition requires $x+y$ to be a positive odd integer. The truncated parallelogram configuration comprises $y+1$ horizontal rows of triangular chambers. The extremal (top and bottom) rows each contain $x-y+2$ triangular chambers, while intermediate rows contain $x-y+3$ chambers. Through this stratified counting approach, we establish the total chamber count in the bi-truncated parallelogram:
\begin{equation}\label{coordinateformula1}
    |\mathrm{Conv}(u,v)| = 2(x-y+2) + (y-1)(x-y+3) = xy + x - y^2 + 2y + 1,
\end{equation}
valid under the parity constraint that $x+y$ is a positive odd integer.

   \begin{figure}[htbp]
    \centering

    \begin{tikzpicture}[shorten >=1pt, node distance=2cm, auto]
            
            \fill[black!20, opacity=1] (2, {sqrt(3)}) -- (1.5, {sqrt(3)/2}) -- (1, {sqrt(3)}) -- cycle;
            \fill[black!20, opacity=1] (4, {2*sqrt(3)}) -- (5, {2*sqrt(3)}) -- (4.5, {5*sqrt(3)/2}) -- cycle;
            \fill[black!20, opacity=0.3] (1, {sqrt(3)}) -- (1.5, {sqrt(3)/2}) -- (3.5, {0.5*sqrt(3)}) -- (5, {2*sqrt(3)}) -- (4.5, {5*sqrt(3)/2}) -- (2.5, {5*sqrt(3)/2}) -- cycle;
            
            \node (1) at (0, 0) {};
            \node (11) at (3, {3*sqrt(3)}) {};
            \node (2) at (1, 0) {};
            \node (12) at (4, {3*sqrt(3)}) {};
            \node (3) at (2, 0) {};
            \node (13) at (5, {3*sqrt(3)}) {};
            \node (4) at (3, 0) {};
            \node (14) at (6, {3*sqrt(3)}) {};
            \node (5) at (4, 0) {};
            \node (23) at (6, {2*sqrt(3)}) {};
            \node (6) at (5, 0) {};
            \node (21) at (6, {sqrt(3)}) {};
            \node (7) at (6, 0) {};
            \node (8) at (0, {3*sqrt(3)}) {};
            \node (9) at (1, {3*sqrt(3)}) {};
            \node (10) at (2, {3*sqrt(3)}) {};
            \node (20) at (6, {0.5*sqrt(3)}) {};
            \node (22) at (6, {1.5*sqrt(3)}) {};
            \node (24) at (6, {2.5*sqrt(3)}) {};
            \node (15) at (0, {0.5*sqrt(3)}) {};
            \node (16) at (0, {sqrt(3)}) {};
            \node (17) at (0, {1.5*sqrt(3)}) {};
            \node (18) at (0, {2*sqrt(3)}) {};
            \node (19) at (0, {2.5*sqrt(3)}) {};
            \node (32) at (1.32, {0.87*sqrt(3)}) {$u$};
            \node (33) at (4.5, {6.5*sqrt(3)/3}) {$v$};
            \node (101) at (0, {0.75*sqrt(3)}) {};
            \node (102) at (6, {0.75*sqrt(3)}) {};
            \node (103) at (1.5, 0) {};
            \node (104) at (1.5, {3*sqrt(3)}) {};
            \node (105) at (6, {0.6*sqrt(3)}) {$x$};
            \node (106) at (1.3, {3*sqrt(3)}) {$y$};

            \draw (1) -- (11);
            \draw (2) -- (12);
            \draw (3) -- (13);
            \draw (4) -- (14);
            \draw (5) -- (23);
            \draw (6) -- (21);
            \draw (7) -- (11);
            \draw (4) -- (8);
            \draw (5) -- (9);
            \draw (6) -- (10);
            \draw (3) -- (18);
            \draw (2) -- (16);
            \draw (16) -- (10);
            \draw (9) -- (18);
            \draw (21) -- (12);
            \draw (23) -- (13);
            \draw (15) -- (20);
            \draw (16) -- (21);
            \draw (17) -- (22);
            \draw (18) -- (23);
            \draw (19) -- (24);
            \draw[->, dashed] (101) -- (102);
            \draw[->, dashed] (103) -- (104);

        \end{tikzpicture}%
    \caption{$u$ is configured as a downward-pointing triangle ($\bigtriangledown$).}
    \label{triangleup}
    \end{figure}

    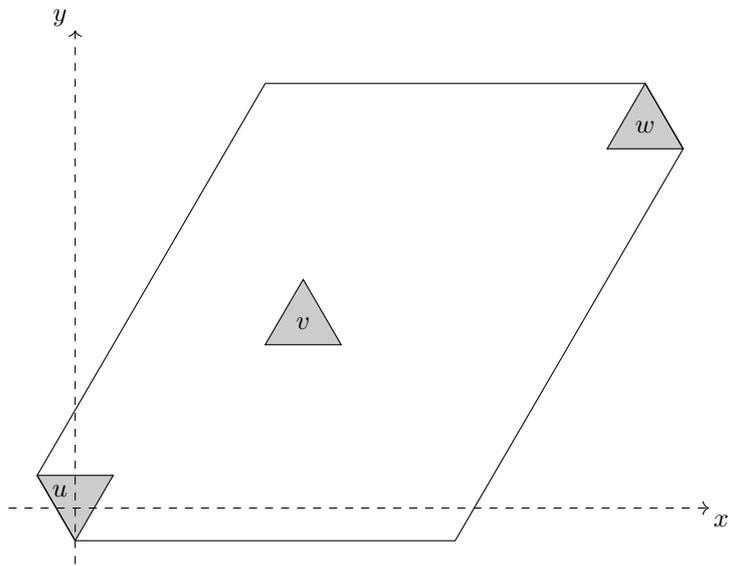
\begin{figure}[htbp]
    \centering

    \begin{tikzpicture}[shorten >=1pt, node distance=2cm, auto]
            
            \fill[black!20] (-1, 0) -- (-1.5, {0.5*sqrt(3)}) -- (-0.5, {0.5*sqrt(3)}) -- cycle;
            \fill[black!20] (2, {2*sqrt(3)}) -- (2.5, {1.5*sqrt(3)}) -- (1.5, {1.5*sqrt(3)}) -- cycle;
            \fill[black!20] (7, {3*sqrt(3)}) -- (6.5, {3.5*sqrt(3)}) -- (6, {3*sqrt(3)}) -- cycle;
            
            \draw (-1, 0) -- (-1.5, {0.5*sqrt(3)}) -- (-0.5, {0.5*sqrt(3)}) -- cycle;
            \draw (2, {2*sqrt(3)}) -- (2.5, {1.5*sqrt(3)}) -- (1.5, {1.5*sqrt(3)}) -- cycle;
            \draw (7, {3*sqrt(3)}) -- (6.5, {3.5*sqrt(3)}) -- (6, {3*sqrt(3)}) -- cycle;
            \draw (-1.5, {0.5*sqrt(3)}) -- (-1, 0) -- (4,0) -- (7, {3*sqrt(3)}) -- (6.5, {3.5*sqrt(3)}) -- (1.5,{3.5*sqrt(3)}) -- cycle;
            \node at (-1.2, {0.55*sqrt(3)-sqrt(3)/6}) {$u$};
            \node at (2, {1.5*sqrt(3)+sqrt(3)/6}) {$v$};
            \node at (6.5, {3*sqrt(3)+sqrt(3)/6}) {$w$};
            \node (101) at (-2, {0.25*sqrt(3)}) {};
            \node (102) at (7.5, {0.25*sqrt(3)}) {};
            \node (103) at (-1, {-0.25*sqrt(3)}) {};
            \node (104) at (-1, {4*sqrt(3)}) {};
            \node at (7.5, {0.15*sqrt(3)}) {$x$};
            \node at (-1.2, {4*sqrt(3)}) {$y$};

            \draw[->, dashed] (101) -- (102);
            \draw[->, dashed] (103) -- (104);
            
        \end{tikzpicture}%
    \caption{Convex hull of $u$, $v$, $w$, where $v$ is configured as a upward-pointing triangle ($\bigtriangleup$)}
    \label{final A2}
    \end{figure}

It suffices to analyze the configuration presented in Fig.~\ref{final A2}, where the vertex $v$ assumes the $\bigtriangleup$ configuration. For configurations where $v$ is a $\bigtriangledown$, we may perform a rotational transformation that maps $v$ to a $\bigtriangleup$ configuration, thereby inducing an equivalence to the case illustrated in Fig.~\ref{final A2}. Let us establish coordinate assignments: $u = (0,0)$, $v = (x,y)$, and $w = (x+a,y+b)$, with parity conditions $x+y$ being a positive odd integer and $a+b$ a positive even integer. The geometric configuration must further satisfy non-crossing constraints relative to the left and right walls of the convex hull $\mathrm{Conv}(u,v,w)$. Violations of these constraints necessitate application of the reduction procedure, yielding a bi-truncated parallelogram configuration as shown in Fig.~\ref{final A2}. This geometric constraint induces the inequalities $x \geq y - 1$ and $a \geq b - 1$. Building upon the cardinality formulas \eqref{coordinateformula2} and \eqref{coordinateformula1}, the verification of inequality \eqref{strong hull property} for $\widetilde{A}_2$ reduces to establishing the following fundamental proposition: 

\begin{proposition}\label{finalprop A2}
    Let non-negative integers $x$, $y$, $a$, $b$ satisfy:
    \begin{enumerate}[(i)]
        \item $x \geq y - 1$ and $a \geq b - 1$,
        \item $x + y$ is a positive odd integer,
        \item $a + b$ is a positive even integer,
    \end{enumerate}
    Then the inequality
    \begin{equation}\label{finalineq}
        \begin{aligned}
        \big(xy &+ x - y^2 + 2y + 1\big)\big(ab + a - b^2 + b + 1\big) \\
        &\geq \Big[(x+a)(y+b) + (x+a) - (y+b)^2 + 2(y+b) + 1\Big]
        \end{aligned}
    \end{equation}
    holds under the given constraints.
\end{proposition}

\begin{proof}
    We begin by observing the elementary inequality $mn \geq m + n$ for all integers $m, n \geq 2$. Given that $a+b$ is an even number, it follows that $a-b$ must also be even and non-negative. Through direct computation, the left-hand side of inequality \eqref{finalineq} expands to:
    
    \begin{equation}
        \begin{aligned}
            \mathrm{LHS}=&(x-y+1)(a-b+2)(y+1)b+2y(a-b+2)(b+1)\\
            &+(x-y+1)(a-b+1)(y+1)-2(y+1)+2.
        \end{aligned}
    \end{equation}
    
    Similarly, direct calculation yields the right-hand side:
    
    \begin{equation}
        \mathrm{RHS}=(x-y+a-b+3)(y+b+1)-2.
    \end{equation}
    
    The proof proceeds by case analysis:
    
    \begin{enumerate}
        \item When $x-y \geq 1$: Applying the multiplicative bound
        \begin{equation}
            (x-y+1)(a-b+2)(y+1)b\geq(x-y+a-b+3)(y+b+1)
        \end{equation}
        in conjunction with
        \begin{equation}
            (x-y+1)(a-b+1)(y+1)\geq 2(y+1),
        \end{equation}
        we derive the critical estimate:
        \begin{equation}
            \begin{aligned}
                \mathrm{LHS}&\geq(x+y+a-b+3)(y+b+1)+2y(a-b+2)(b+1)+2\\
                &\geq (x-y+a-b+3)(y+b+1)-2=\mathrm{RHS}.
            \end{aligned}
        \end{equation}
        
        \item When $x-y = -1$: Here $y$ must be positive. When $y=1$, the inequality
        \begin{equation}
            2y(a-b+2)(b+1)>(a-b+2)(b+2)
        \end{equation}
        combined with the constant term $-2$ gives
        \begin{equation}
            \mathrm{LHS}>(a-b+2)(b+2)-2=\mathrm{RHS}.
        \end{equation}
        
        For $y \geq 2$, we analyze the residual quantity:
        \begin{equation}
            \begin{aligned}
                &\mathrm{LHS}-\mathrm{RHS}+2(y+1)-4\\
                =&2y(a-b+2)(b+1)-(a-b+2)(y+b+1)\\
                =&(a-b+2)(2by+y-b-1).
            \end{aligned}
        \end{equation}
        
        \begin{itemize}
            \item \textit{Subcase~1:} $b \geq 1$. Employing the bound $by \geq 2b \geq b+1$, we obtain
            \begin{equation}
                (a-b+2)(2by+y-b-1) \geq 2(y+1)>2(y+1)-4.
            \end{equation}
            
            \item \textit{Subcase~2} $b=0$. Direct substitution produces
            \begin{equation}
                (a+2)(y-1) \geq 2y - 2 = 2(y+1)-4.
            \end{equation}
        \end{itemize}
    \end{enumerate}
    
    The case-by-case verification in the above establishes the validity of inequality \eqref{finalineq}.
\end{proof}

Therefore, we have established Conjecture~\ref{strong hull conj} for Coxeter groups of type $\widetilde{A}_2$.

\begin{theorem}[Strong hull property for type $\widetilde{A}_2$]\label{A2 strong hull property}
    The Cayley graph of affine type $\widetilde{A}_2$ has the strong hull property.
\end{theorem}

\subsection{Affine type \texorpdfstring{$\widetilde{C}_2$}{C2}}\label{affine type C2}

The corresponding Cayley graph of $\widetilde{C}_2$ has already been discussed in Fig.~\ref{fig:reflection hyperplanes in Euclidean plane} and at the end of Section~\ref{Background}. Here, we apply reduction techniques for classification and computation, by methodology analogous to that employed in Section~\ref{affine type A2}. The reduction techniques employed differ in implementation from those applied to $\widetilde{A}_2$ configurations. We shall demonstrate this process by three examples.

\begin{example}\label{C2 example 1}
    Consider the geometric structure of the convex hull of $u$, $v$, and $w$ illstrated in Fig.~\ref{C2 example Translate $u$ to $u_1$}.

    \begin{figure}[htbp]
    \centering
    \resizebox{\textwidth}{!}{

    \begin{tikzpicture}[shorten >=1pt, node distance=2cm, auto]
            
            \fill[black!20] (0,0) -- (1,0) -- (0,1) -- cycle;
            \fill[black!20] (1,5) -- (2,5) -- (1,6) -- cycle;
            \fill[black!20] (9,2) -- (10,3) -- (9,3) -- cycle;
            \fill[black!20] (-4,0) -- (-3,0) -- (-4,1) -- cycle;
            \fill[black!20] (12,5) -- (13,6) -- (12,6) -- cycle;
            \fill[black!20, opacity=0.3] (0,0) -- (1,0) -- (2,1) -- (2,5) -- (1,6) -- (0,5) -- cycle;
            \fill[black!20, opacity=0.3] (1,6) -- (1,4) -- (3,2) -- (9,2) -- (10,3) -- (7,6) -- cycle;
            \fill[black!20, opacity=0.3] (-4,0) -- (-3,0) -- (2,5) -- (1,6) -- (-4,1) -- cycle;
            \fill[black!20, opacity=0.3] (1,6) -- (1,5) -- (12,5) -- (13,6) -- cycle;
            \fill[pattern=sparse bold dots, opacity=0.3] (0,5) -- (2,5) -- (2,1) -- (0,1) -- cycle;
            \fill[pattern=sparse bold dots, opacity=0.3] (0,5) -- (2,5) -- (-2,1) -- (-4,1) -- cycle;
            \fill[pattern=sparse bold lines, opacity=0.3] (2,5) -- (3,6) -- (10,3) -- (9,2) -- cycle;
            \fill[pattern=sparse bold lines, opacity=0.3] (2,5) -- (3,6) -- (13,6) -- (12,5) -- cycle;
            \draw (0,0) -- (1,0) -- (0,1) -- cycle;
            \draw (1,5) -- (2,5) -- (1,6) -- cycle;
            \draw (9,2) -- (10,3) -- (9,3) -- cycle;
            \draw (0,0) -- (7,0) -- (10,3) -- (7,6) -- (1,6) -- (0,5) -- cycle;
            \draw (-4,0) -- (-3,0) -- (-4,1) -- cycle;
            \draw (12,5) -- (13,6) -- (12,6) -- cycle;
            \draw (-4,1) -- (-4,0) -- (7,0) -- (13,6) -- (1,6) -- cycle;
            \draw[dashed] (0,5) -- (2,5) -- (2,1) -- (0,1) -- cycle;
            \draw[dashed] (2,5) -- (3,6) -- (10,3) -- (9,2) -- cycle;
            \draw[dashed] (0,5) -- (2,5) -- (-2,1) -- (-4,1) -- cycle;
            \draw[dashed] (2,5) -- (3,6) -- (13,6) -- (12,5) -- cycle;
            \node at (0.3,0.3) {$u$};
            \node at (1.3,5.3) {$v$};
            \node at (9.3,2.7) {$w$};
            \node at (-3.7,0.3) {$u_1$};
            \node at (12.3,5.7) {$w_1$};

        \end{tikzpicture}%
        }
    \caption{Translate $u$ to $u_1$, and $w$ to $w_1$}
    \label{C2 example Translate $u$ to $u_1$}
    \end{figure}
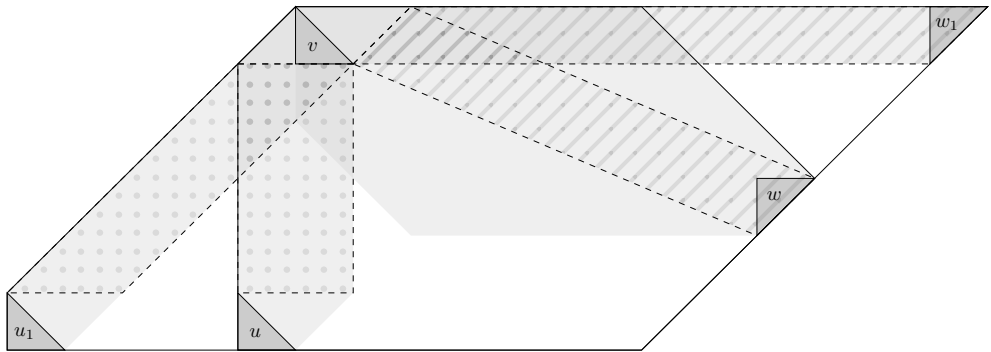

    The treatment of the case $\widetilde{C}_2$ differs from that in Section~\ref{affine type A2} in that direct estimation through the strong hull inequality~\eqref{strong hull property} by analyzing individual chambers within a gallery becomes intractable. Instead, we employ a geometric approach by interpreting each chamber as a unit area and utilizing the fundamental property that the base length and height determine the area of a parallelogram.

    As illustrated in Fig.~\ref{C2 example Translate $u$ to $u_1$}, translating $u$ leftward to $u_1$ allows comparison between $\mathrm{Conv}(u,v)$ and $\mathrm{Conv}(u_1,v)$. Through auxiliary constructions shown in the diagram, we observe that the quadrilateral regions shaded with dots in both $\mathrm{Conv}(u,v)$ and $\mathrm{Conv}(u_1,v)$ possess equal areas. This geometric equivalence implies that the cardinalities of chamber sets contained in the dotted regions are identical, thereby establishing $|\mathrm{Conv}(u,v)|=|\mathrm{Conv}(u_1,v)|$. Similarly, translating $w$ northeast to $w_1$ enables comparison between $\mathrm{Conv}(v,w)$ and $\mathrm{Conv}(v,w_1)$. The hatched quadrilateral regions bounded by auxiliary lines demonstrate equal areas. Let $\overline{\mathrm{Conv}}(w)$ denote the set of chambers that consists of the hatched quadrilateral region intersecting the interior of the convex hull that corresponds to $w$. We derive the inequality:
\begin{equation}\label{C2 ineq w->w1}
    |\mathrm{Conv}(v,w_1)|=3+|\overline{\mathrm{Conv}}(w_1)|\leq 3+|\overline{\mathrm{Conv}}(w)|\leq|\mathrm{Conv}(v,w)|.
\end{equation}
Furthermore, the translation process yields the monotonicity relation:
\begin{equation}\label{C2 ineq u1 v w1}
    |\mathrm{Conv}(u_1,v,w_1)|\geq|\mathrm{Conv}(u,v,w)|.
\end{equation}
To establish the strong hull inequality~\eqref{strong hull property} in this configuration, it suffices to verify:
\begin{equation}\label{C2 ineq u1 v v w1 u1 v w1}
    |\mathrm{Conv}(u_1,v)|\cdot|\mathrm{Conv}(v,w_1)|\geq|\mathrm{Conv}(u_1,v,w_1)|.
\end{equation}

    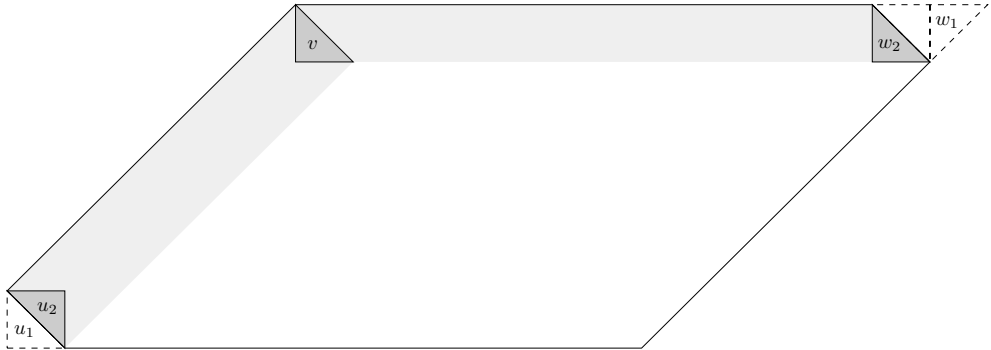
\begin{figure}[htbp]
    \centering
    \resizebox{\textwidth}{!}{

    \begin{tikzpicture}[shorten >=1pt, node distance=2cm, auto]
            
            \fill[black!20] (1,5) -- (2,5) -- (1,6) -- cycle;
            \fill[black!20] (-4,1) -- (-3,0) -- (-3,1) -- cycle;
            \fill[black!20] (11,5) -- (12,5) -- (11,6) -- cycle;
            \fill[black!20, opacity=0.3] (-4,1) -- (-3,0) -- (2,5) -- (1,6) -- cycle;
            \fill[black!20, opacity=0.3] (1,6) -- (1,5) -- (12,5) -- (11,6) -- cycle;
            \draw (1,5) -- (2,5) -- (1,6) -- cycle;
            \draw (-4,1) -- (-3,0) -- (-3,1) -- cycle;
            \draw (11,5) -- (12,5) -- (11,6) -- cycle;
            \draw (-4,1) -- (-3,0) -- (7,0) -- (12,5) -- (11,6) -- (1,6) -- cycle;
            \draw[dashed] (-4,0) -- (-3,0) -- (-4,1) -- cycle;
            \draw[dashed] (11,6) -- (12,5) -- (13,6) -- cycle;
            \draw[dashed] (12,6) -- (12,5) -- cycle;
            \node at (1.3,5.3) {$v$};
            \node at (-3.7,0.3) {$u_1$};
            \node at (12.3,5.7) {$w_1$};
            \node at (-3.3,0.7) {$u_2$};
            \node at (11.3,5.3) {$w_2$};

        \end{tikzpicture}%
        }
    \caption{Reduction of Example~\ref{C2 example 1}}
    \label{C2 example final reduced case}
    \end{figure}

Following the inequalities~\eqref{base inequality} and~\eqref{case2 u3w1 A2}, we apply reflection actions (potentially through multiple iterations) to the configurations in Fig.~\ref{C2 example Translate $u$ to $u_1$}, thereby transforming them into the geometric arrangement depicted in Fig.~\ref{C2 example final reduced case}. Analogous to the argumentation in Section~\ref{affine type A2}, the verification of inequality~\eqref{C2 ineq u1 v v w1 u1 v w1} reduces to establishing the following cardinality inequality:
\begin{equation}\label{C2 ineq u2 v v w2 u2 v w2}
    |\mathrm{Conv}(u_2,v)|\cdot|\mathrm{Conv}(v,w_2)|\geq|\mathrm{Conv}(u_2,v,w_2)|.
\end{equation}

\end{example}

\begin{example}\label{C2 example 2}
    Consider the geometric structure of the convex hull of $u$, $v$, and $w$ illstrated in Fig.~\ref{C2 example 2 Translate $u$ to $u_1$}. Employing methodology analogous to Example~\ref{C2 example 1}, we translate elements $u$ and $w$ to positions $u_1$ and $w_1$, respectively, as depicted in Fig.~\ref{C2 example 2 Translate $u$ to $u_1$}. It is noteworthy that $u_1$ may share the same orientation as $u$. In such cases, the analysis of corresponding convex hulls follows directly through the established methodology. However, to distinguish from previous discussions, we specifically verify the scenario where $u$ and $u_1$ exhibit distinct orientations. The core methodology remains consistent regardless of orientation parity.

    \begin{figure}[htbp]
    \centering

    \begin{tikzpicture}[shorten >=1pt, node distance=2cm, auto]
            
            \fill[black!20] (3,0) -- (4,1) -- (3,1) -- cycle;
            \fill[black!20] (0,6) -- (0,5) -- (1,6) -- cycle;
            \fill[black!20] (8,2) -- (8,3) -- (7,2) -- cycle;
            \fill[black!20] (0,1) -- (1,0) -- (1,1) -- cycle;
            \fill[black!20] (11,5) -- (11,6) -- (10,5) -- cycle;
            \fill[black!20, opacity=0.3] (1,6) -- (0,6) -- (0,3) -- (3,0) -- (4,1) -- (4,3) -- cycle;
            \fill[black!20, opacity=0.3] (0,6) -- (0,5) -- (3,2) -- (8,2) -- (8,3) -- (5,6) -- cycle;
            \fill[black!20, opacity=0.3] (0,6) -- (0,1) -- (1,0) -- (1,6) -- cycle;
            \fill[black!20, opacity=0.3] (0,6) -- (0,5) -- (11,5) -- (11,6) -- cycle;
            \fill[pattern=sparse bold dots, opacity=0.3] (0,5) -- (3,1) -- (4,1) -- (1,5) -- cycle;
            \fill[pattern=sparse bold dots, opacity=0.3] (0,5) -- (0,1) -- (1,1) -- (1,5) -- cycle;
            \fill[pattern=sparse bold lines, opacity=0.3] (0,5) -- (7,2) -- (8,3) -- (1,6) -- cycle;
            \fill[pattern=sparse bold lines, opacity=0.3] (0,5) -- (10,5) -- (11,6) -- (1,6) -- cycle;
            \draw (3,0) -- (4,1) -- (3,1) -- cycle;
            \draw (0,6) -- (0,5) -- (1,6) -- cycle;
            \draw (8,2) -- (8,3) -- (7,2) -- cycle;
            \draw (0,1) -- (1,0) -- (1,1) -- cycle;
            \draw (11,5) -- (11,6) -- (10,5) -- cycle;
            \draw (0,6) -- (0,3) -- (3,0) -- (7,0) -- (8,1) -- (8,3) -- (5,6) -- cycle;
            \draw (0,6) -- (0,1) -- (1,0) -- (7,0) -- (11,4) -- (11,6) -- cycle;
            \draw[dashed] (0,5) -- (3,1) -- (4,1) -- (1,5) -- cycle;
            \draw[dashed] (0,5) -- (0,1) -- (1,0) -- (1,5) -- cycle;
            \draw[dashed] (0,5) -- (7,2) -- (8,3) -- (1,6) -- cycle;
            \draw[dashed] (0,5) -- (10,5) -- (11,6) -- (1,6) -- cycle;
            \draw[dashed] (8,3) -- (10,5);
            \draw[dashed] (11,4) -- (10,5);
            \node at (0.3,5.7) {$v$};
            \node at (0.7,0.7) {$u_1$};
            \node at (3.3,0.7) {$u$};
            \node at (7.7,2.3) {$w$};
            \node at (10.7,5.3) {$w_1$};

        \end{tikzpicture}%
        
    \caption{Translate $u$ to $u_1$, and $w$ to $w_1$}
    \label{C2 example 2 Translate $u$ to $u_1$}
    \end{figure}
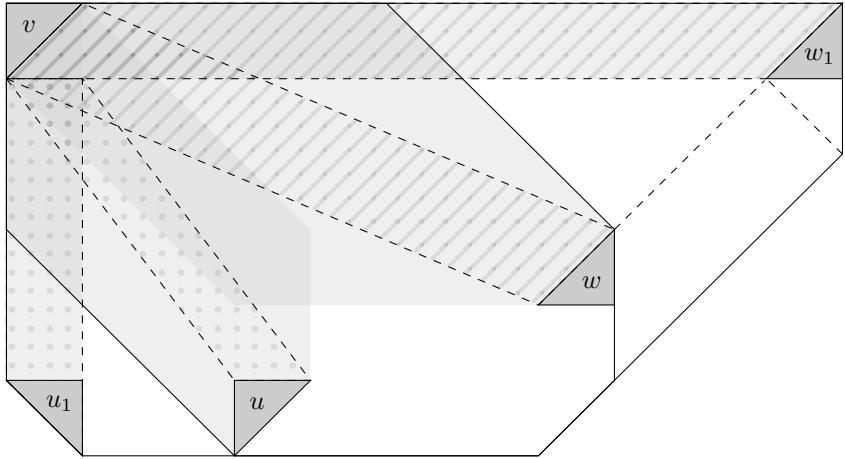

For the translation of $u$ to $u_1$ through iterative reflections, we establish the cardinality relation:
\begin{equation}\label{C2 example ineq u->u1}
    |\mathrm{Conv}(u_1,v)|=3+|\overline{\mathrm{Conv}}(u_1)|\leq3+|\overline{\mathrm{Conv}}(u)|\leq |\mathrm{Conv}(u,v)|.
\end{equation}
Similarly, for the transformation of $w$ to $w_1$, we derive:
\begin{equation}\label{C2 example ineq w->w1}
    |\mathrm{Conv}(v,w_1)|=3+|\overline{\mathrm{Conv}}(w_1)|\leq3+|\overline{\mathrm{Conv}}(w)|\leq |\mathrm{Conv}(v,w)|.
\end{equation}
Crucially, this translation process preserves the monotonicity of convex hull cardinalities:
\begin{equation}\label{C2 example u w u1 w1}
    |\mathrm{Conv}(u_1,v,w_1)|\geq |\mathrm{Conv}(u,v,w)|.
\end{equation}
Through systematic application of inequalities~\eqref{C2 example ineq u->u1},~\eqref{C2 example ineq w->w1}, and~\eqref{C2 example u w u1 w1}, the proof of our main theorem reduces to demonstrating:
\begin{equation}\label{C2 example ineq u1 v v w1 u1 v w1}
    |\mathrm{Conv}(u_1,v)|\cdot|\mathrm{Conv}(v,w_1)|\geq|\mathrm{Conv}(u_1,v,w_1)|.
\end{equation}
In fact, we may further reflect element $v$ to position $v_1$, as illustrated in Fig.~\ref{C2 example 2 final figure}. The critical inequality follows from:
\begin{equation}\label{C2 exmaple 2 final ineq}
\begin{aligned}
    |\mathrm{Conv}(u_1,v_1)|\cdot|\mathrm{Conv}(v_1,w_1)| &= (|\mathrm{Conv}(u_1,v)|-1)\cdot(|\mathrm{Conv}(v,w_1)|-1) \\
    &\geq (|\mathrm{Conv}(u_1,v,w_1)|-1) = |\mathrm{Conv}(u_1,v_1,w_1)|.
\end{aligned}
\end{equation}
This establishes the required implication for inequality~\eqref{C2 example ineq u1 v v w1 u1 v w1}. Consequently, the proof for this configuration reduces to verifying inequality~\eqref{C2 exmaple 2 final ineq}, thereby completing the reduction argument.

    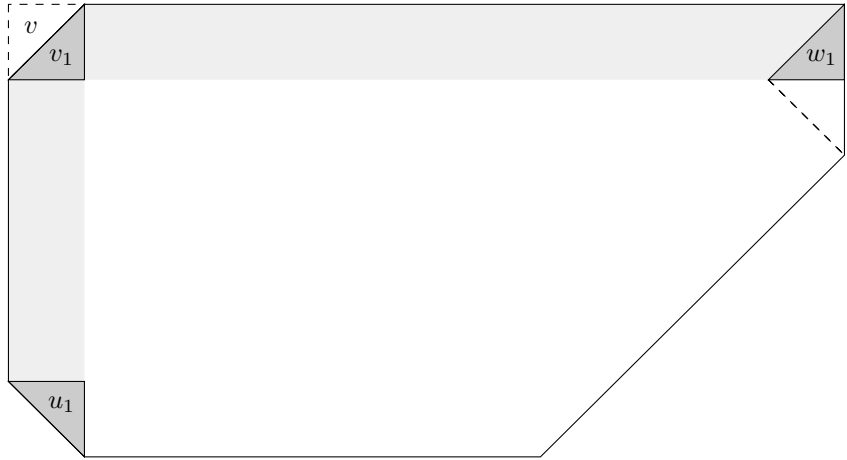
\begin{figure}[htbp]
    \centering

    \begin{tikzpicture}[shorten >=1pt, node distance=2cm, auto]
            
            \fill[black!20] (0,1) -- (1,0) -- (1,1) -- cycle;
            \fill[black!20] (0,5) -- (1,5) -- (1,6) -- cycle;
            \fill[black!20] (10,5) -- (11,5) -- (11,6) -- cycle;
            \fill[black!20, opacity=0.3] (1,6) -- (0,5) -- (0,1) -- (1,0) -- cycle;
            \fill[black!20, opacity=0.3] (1,6) -- (0,5) -- (11,5) -- (11,6) -- cycle;
            \draw (0,1) -- (1,0) -- (1,1) -- cycle;
            \draw (0,5) -- (1,5) -- (1,6) -- cycle;
            \draw (10,5) -- (11,5) -- (11,6) -- cycle;
            \draw (1,6) -- (0,5) -- (0,1) -- (1,0) -- (7,0) -- (11,4) -- (11,6) -- cycle;
            \draw[dashed] (0,6) -- (0,5) -- (1,6) -- cycle;
            \draw[dashed] (10,5) -- (11,4) -- cycle;
            \node at (0.3,5.7) {$v$};
            \node at (0.7,5.3) {$v_1$};
            \node at (0.7,0.7) {$u_1$};
            \node at (10.7,5.3) {$w_1$};

        \end{tikzpicture}%
        
    \caption{Reduction of Example~\ref{C2 example 2}}
    \label{C2 example 2 final figure}
    \end{figure}
\end{example}

\begin{example}\label{C2 exmaple 3}
    If the chamber $u$ in Fig.~\ref{C2 example 2 Translate $u$ to $u_1$} lies to the lower-left of $v$ rather than the lower-right, we apply distinct reduction techniques. To ensure a representative range of examples, we adjust the position of $u$; this modification is justified by the final reduced configuration in Fig.~\ref{C2 example 3 reduced} and the subsequent complete classification for $\widetilde{C}_2$.

    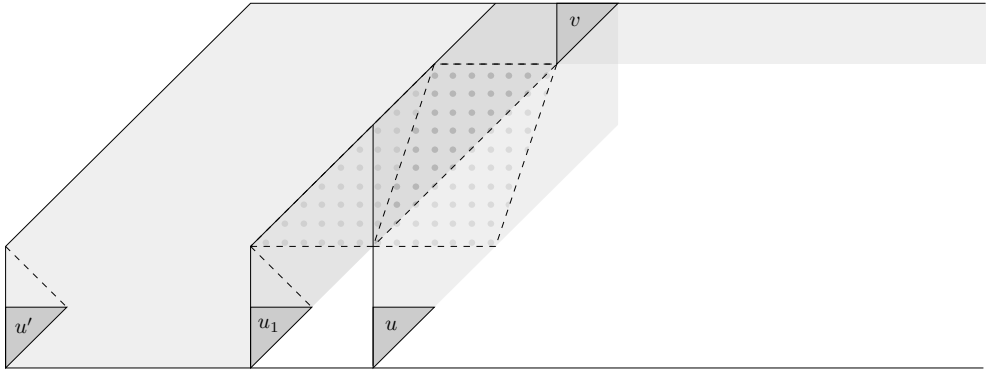
\begin{figure}[htbp]
    \centering
    \resizebox{\textwidth}{!}{

    \begin{tikzpicture}[shorten >=1pt, node distance=2cm, auto]
            
            \fill[black!20] (-3,0) -- (-3,1) -- (-2,1) -- cycle;
            \fill[black!20] (0,5) -- (0,6) -- (1,6) -- cycle;
            \fill[black!20] (-5,0) -- (-5,1) -- (-4,1) -- cycle;
            \fill[black!20] (-9,0) -- (-9,1) -- (-8,1) -- cycle;
            \fill[black!20, opacity=0.3] (-3,0) -- (1,4) -- (1,6) -- (-1,6) -- (-3,4) -- cycle;
            \fill[black!20, opacity=0.3] (0,6) -- (0,5) -- (7,5) -- (7,6) -- cycle;
            \fill[black!20, opacity=0.3] (-5,0) -- (1,6) -- (-1,6) -- (-5,2) -- cycle;
            \fill[black!20, opacity=0.3] (-9,0) -- (-5,0) -- (1,6) -- (-5,6) -- (-9,2) -- cycle;
            \fill[pattern=sparse bold dots, opacity=0.3] (-3,2) -- (-1,2) -- (0,5) -- (-2,5) -- cycle;
            \fill[pattern=sparse bold dots, opacity=0.3] (-5,2) -- (-3,2) -- (0,5) -- (-2,5) -- cycle;
            \draw (-3,0) -- (-3,1) -- (-2,1) -- cycle;
            \draw (0,5) -- (0,6) -- (1,6) -- cycle;
            \draw (-5,0) -- (-5,1) -- (-4,1) -- cycle;
            \draw (-9,0) -- (-9,1) -- (-8,1) -- cycle;
            \draw (7,6) -- (-5,6) -- (-9,2) -- (-9,0) -- (7,0);
            \draw (-1,6) -- (-5,2) -- (-5,0);
            \draw (-3,4) -- (-3,0);
            \draw[dashed] (-3,2) -- (-1,2) -- (0,5) -- (-2,5) -- cycle;
            \draw[dashed] (-5,2) -- (-3,2) -- (0,5) -- (-2,5) -- cycle;
            \draw[dashed] (-5,2) -- (-4,1);
            \draw[dashed] (-9,2) -- (-8,1);
            \node at (0.3,5.7) {$v$};
            \node at (-2.7,0.7) {$u$};
            \node at (-8.7,0.7) {$u'$};
            \node at (-4.7,0.7) {$u_1$};

        \end{tikzpicture}%
        }
    \caption{Reduction techniques for different positions of the chamber of $u$}
    \label{C2 example 3 reduced}
    \end{figure}

    Translating $u$ to $u_1$ via iterative reflections (Fig.~\ref{C2 example 3 reduced}) and employing the auxiliary line construction from Examples~\ref{C2 example 1} and~\ref{C2 example 2} yields the inequalities
\begin{equation}
|\mathrm{Conv}(u_1,v)| \leq |\mathrm{Conv}(u,v)| \quad \text{and} \quad |\mathrm{Conv}(u_1,v,w_1)| \leq |\mathrm{Conv}(u,v,w_1)|.
\end{equation}

Suppose $u$ is instead located at $u'$ in Fig.~\ref{C2 example 3 reduced}. By applying multiple reflections to translate it to $u_1$, we observe that  
\begin{equation}  
|\mathrm{Conv}(u',v)| - |\mathrm{Conv}(u_1,v)| = |\mathrm{Conv}(u',v,w_1)| - |\mathrm{Conv}(u_1,v,w_1)|.  
\end{equation}  

In either scenario mentioned above, to establish the strong hull inequality~\eqref{strong hull property}, it suffices to verify the inequality  
\begin{equation}  
|\mathrm{Conv}(u_1,v)| \cdot |\mathrm{Conv}(v,w_1)| \geq |\mathrm{Conv}(u_1,v,w_1)|.  
\end{equation}  
\end{example}

In Example~\ref{C2 exmaple 3}, $u$ may lie in the exterior of the reduced hull. While Example~\ref{C2 example 1} focused on the interior case, exterior chambers can be translated inward analogously. Consequently, arbitrary convex hulls reduce to the canonical configurations classified in Fig.~\ref{fig:all_C2_cases}.

\begin{figure}[htbp]
    \centering
    \noindent
    %
    \begin{minipage}[b]{0.31\textwidth}
        \centering
        \subfloat[Case 1 for $\widetilde{C}_2$]{%
        \label{fig:C2case1}
        \resizebox{\linewidth}{!}{%
            \begin{tikzpicture}[shorten >=1pt, baseline=(current bounding box.center)]
            \fill[black!20] (0,1) -- (1,0) -- (1,1) -- cycle;
            \fill[black!20] (9,5) -- (9,6) -- (10,5) -- cycle;
            \fill[black!20] (4,5) -- (5,5) -- (5,6) -- cycle;
            \fill[black!20] (6,5) -- (5,5) -- (5,6) -- cycle;
            \fill[black!20] (6,5) -- (6,6) -- (5,6) -- cycle;
            \fill[black!20] (6,5) -- (6,6) -- (7,6) -- cycle;
            \fill[black!20] (4,2) -- (4,3) -- (5,3) -- cycle;
            \draw (0,1) -- (1,0) -- (5,0) -- (10,5) -- (9,6) -- (5,6) -- cycle;
            \draw (0,1) -- (1,0) -- (1,1) -- cycle;
            \draw (9,5) -- (9,6) -- (10,5) -- cycle;
            \draw (4,5) -- (5,5) -- (5,6) -- cycle;
            \draw (6,5) -- (5,5) -- (5,6) -- cycle;
            \draw (6,5) -- (6,6) -- (5,6) -- cycle;
            \draw (6,5) -- (6,6) -- (7,6) -- cycle;
            \draw (4,2) -- (4,3) -- (5,3) -- cycle;
            \node at (0.7,0.7) {$u$};
            \node at (9.3,5.3) {$w$};
            \node at (4.7,5.3) {$v$};
            \node at (5.3,5.3) {$v$};
            \node at (5.7,5.7) {$v$};
            \node at (6.3,5.7) {$v$};
            \node at (4.3,2.7) {$v$};
            \end{tikzpicture}%
        }}
    \end{minipage}
    \hfill
    \begin{minipage}[b]{0.31\textwidth}
        \centering
        \subfloat[Case 2 for $\widetilde{C}_2$]{%
        \label{fig:C2case2}
        \resizebox{\linewidth}{!}{%
            \begin{tikzpicture}[shorten >=1pt, baseline=(current bounding box.center)]
            \fill[black!20] (0,1) -- (1,0) -- (1,1) -- cycle;
            \fill[black!20] (9,5) -- (10,6) -- (10,5) -- cycle;
            \fill[black!20] (4,5) -- (5,5) -- (5,6) -- cycle;
            \fill[black!20] (6,5) -- (5,5) -- (5,6) -- cycle;
            \fill[black!20] (6,5) -- (6,6) -- (5,6) -- cycle;
            \fill[black!20] (6,5) -- (6,6) -- (7,6) -- cycle;
            \fill[black!20] (5,2) -- (5,3) -- (6,3) -- cycle;
            \draw (0,1) -- (1,0) -- (6,0) -- (10,4) -- (10,6) -- (5,6) -- cycle;
            \draw (0,1) -- (1,0) -- (1,1) -- cycle;
            \draw (9,5) -- (10,6) -- (10,5) -- cycle;
            \draw (4,5) -- (5,5) -- (5,6) -- cycle;
            \draw (6,5) -- (5,5) -- (5,6) -- cycle;
            \draw (6,5) -- (6,6) -- (5,6) -- cycle;
            \draw (6,5) -- (6,6) -- (7,6) -- cycle;
            \draw (5,2) -- (5,3) -- (6,3) -- cycle;
            \draw[dashed] (9,5) -- (10,4);
            \node at (0.7,0.7) {$u$};
            \node at (9.7,5.3) {$w$};
            \node at (4.7,5.3) {$v$};
            \node at (5.3,5.3) {$v$};
            \node at (5.7,5.7) {$v$};
            \node at (6.3,5.7) {$v$};
            \node at (5.3,2.7) {$v$};
            \end{tikzpicture}%
        }}
    \end{minipage}
    \hfill
    \begin{minipage}[b]{0.31\textwidth}
        \centering
        \subfloat[Case 3 for $\widetilde{C}_2$]{%
        \label{fig:C2case3}
        \resizebox{\linewidth}{!}{%
            \begin{tikzpicture}[shorten >=1pt, baseline=(current bounding box.center)]
            \fill[black!20] (2,1) -- (1,0) -- (1,1) -- cycle;
            \fill[black!20] (11,5) -- (11,6) -- (10,5) -- cycle;
            \fill[black!20] (4,5) -- (5,5) -- (5,6) -- cycle;
            \fill[black!20] (6,5) -- (5,5) -- (5,6) -- cycle;
            \fill[black!20] (6,5) -- (6,6) -- (5,6) -- cycle;
            \fill[black!20] (6,5) -- (6,6) -- (7,6) -- cycle;
            \fill[black!20] (6,2) -- (6,3) -- (7,3) -- cycle;
            \draw (1,2) -- (1,0) -- (7,0) -- (11,4) -- (11,6) -- (5,6) -- cycle;
            \draw (2,1) -- (1,0) -- (1,1) -- cycle;
            \draw (11,5) -- (11,6) -- (10,5) -- cycle;
            \draw (4,5) -- (5,5) -- (5,6) -- cycle;
            \draw (6,5) -- (5,5) -- (5,6) -- cycle;
            \draw (6,5) -- (6,6) -- (5,6) -- cycle;
            \draw (6,5) -- (6,6) -- (7,6) -- cycle;
            \draw (6,2) -- (6,3) -- (7,3) -- cycle;
            \draw[dashed] (10,5) -- (11,4);
            \draw[dashed] (1,2) -- (2,1);
            \node at (1.3,0.7) {$u$};
            \node at (10.7,5.3) {$w$};
            \node at (4.7,5.3) {$v$};
            \node at (5.3,5.3) {$v$};
            \node at (5.7,5.7) {$v$};
            \node at (6.3,5.7) {$v$};
            \node at (6.3,2.7) {$v$};
            \end{tikzpicture}%
        }}
    \end{minipage}

    \vspace{0.5cm}
    \begin{minipage}[b]{0.31\textwidth}
        \centering
        \subfloat[Case 4 for $\widetilde{C}_2$]{%
        \label{fig:C2case4}
        \resizebox{\linewidth}{!}{%
            \begin{tikzpicture}[shorten >=1pt, baseline=(current bounding box.center)]
            \fill[black!20] (0,1) -- (1,0) -- (1,1) -- cycle;
            \fill[black!20] (0,5) -- (1,5) -- (1,6) -- cycle;
            \fill[black!20] (9,5) -- (9,6) -- (10,5) -- cycle;
            \draw (1,6) -- (0,5) -- (0,1) -- (1,0) -- (5,0) -- (10,5) -- (9,6) -- cycle;
            \draw (0,1) -- (1,0) -- (1,1) -- cycle;
            \draw (0,5) -- (1,5) -- (1,6) -- cycle;
            \draw (9,5) -- (9,6) -- (10,5) -- cycle;
            \node at (0.7,0.7) {$u$};
            \node at (0.7,5.3) {$v$};
            \node at (9.3,5.3) {$w$};
            \end{tikzpicture}%
        }}
    \end{minipage}
    \hfill
    \begin{minipage}[b]{0.31\textwidth}
        \centering
        \subfloat[Case 5 for $\widetilde{C}_2$]{%
        \label{fig:C2case5}
        \resizebox{\linewidth}{!}{%
            \begin{tikzpicture}[shorten >=1pt, baseline=(current bounding box.center)]
            \fill[black!20] (0,1) -- (1,0) -- (1,1) -- cycle;
            \fill[black!20] (0,5) -- (1,5) -- (1,6) -- cycle;
            \fill[black!20] (9,5) -- (10,6) -- (10,5) -- cycle;
            \draw (1,6) -- (0,5) -- (0,1) -- (1,0) -- (6,0) -- (10,4) -- (10,6) -- cycle;
            \draw (0,1) -- (1,0) -- (1,1) -- cycle;
            \draw (0,5) -- (1,5) -- (1,6) -- cycle;
            \draw (9,5) -- (10,6) -- (10,5) -- cycle;
            \draw[dashed] (9,5) -- (10,4);
            \node at (0.7,0.7) {$u$};
            \node at (0.7,5.3) {$v$};
            \node at (9.7,5.3) {$w$};
            \end{tikzpicture}%
        }}
    \end{minipage}
    \hfill
    \begin{minipage}[b]{0.31\textwidth}
        \centering
        \subfloat[Case 6 for $\widetilde{C}_2$]{%
        \label{fig:C2case6}
        \resizebox{\linewidth}{!}{%
            \begin{tikzpicture}[shorten >=1pt, baseline=(current bounding box.center)]
            \fill[black!20] (0,0) -- (0,1) -- (1,1) -- cycle;
            \fill[black!20] (0,5) -- (1,5) -- (1,6) -- cycle;
            \fill[black!20] (9,5) -- (10,6) -- (10,5) -- cycle;
            \draw (1,6) -- (0,5) -- (0,0) -- (6,0) -- (10,4) -- (10,6) -- cycle;
            \draw (0,0) -- (0,1) -- (1,1) -- cycle;
            \draw (0,5) -- (1,5) -- (1,6) -- cycle;
            \draw (9,5) -- (10,6) -- (10,5) -- cycle;
            \draw[dashed] (9,5) -- (10,4);
            \node at (0.3,0.7) {$u$};
            \node at (0.7,5.3) {$v$};
            \node at (9.7,5.3) {$w$};
            \end{tikzpicture}%
        }}
    \end{minipage}

    \vspace{0.5cm}
    \begin{minipage}[b]{0.31\textwidth}
        \centering
        \subfloat[Case 7 for $\widetilde{C}_2$]{%
        \label{fig:C2case7}
        \resizebox{\linewidth}{!}{%
            \begin{tikzpicture}[shorten >=1pt, baseline=(current bounding box.center)]
            \fill[black!20] (0,0) -- (1,0) -- (1,1) -- cycle;
            \fill[black!20] (0,5) -- (1,5) -- (1,6) -- cycle;
            \fill[black!20] (9,5) -- (9,6) -- (10,5) -- cycle;
            \draw (1,6) -- (0,5) -- (0,0) -- (5,0) -- (10,5) -- (9,6) -- cycle;
            \draw (0,0) -- (1,0) -- (1,1) -- cycle;
            \draw (0,5) -- (1,5) -- (1,6) -- cycle;
            \draw (9,5) -- (9,6) -- (10,5) -- cycle;
            \node at (0.7,0.3) {$u$};
            \node at (0.7,5.3) {$v$};
            \node at (9.3,5.3) {$w$};
            \end{tikzpicture}%
        }}
    \end{minipage}
    \hfill
    \begin{minipage}[b]{0.31\textwidth}
        \centering
        \subfloat[Case 8 for $\widetilde{C}_2$]{%
        \label{fig:C2case8}
        \resizebox{\linewidth}{!}{%
            \begin{tikzpicture}[shorten >=1pt, baseline=(current bounding box.center)]
            \fill[black!20] (0,0) -- (1,0) -- (1,1) -- cycle;
            \fill[black!20] (0,5) -- (1,5) -- (0,6) -- cycle;
            \fill[black!20] (0,6) -- (1,6) -- (1,5) -- cycle;
            \fill[black!20] (9,5) -- (9,6) -- (10,5) -- cycle;
            \draw (0,6) -- (0,0) -- (5,0) -- (10,5) -- (9,6) -- cycle;
            \draw (0,0) -- (1,0) -- (1,1) -- cycle;
            \draw (0,5) -- (1,5) -- (0,6) -- cycle;
            \draw (0,6) -- (1,6) -- (1,5) -- cycle;
            \draw (9,5) -- (9,6) -- (10,5) -- cycle;
            \node at (0.7,0.3) {$u$};
            \node at (0.3,5.3) {$v$};
            \node at (0.7,5.7) {$v$};
            \node at (9.3,5.3) {$w$};
            \end{tikzpicture}%
        }}
    \end{minipage}
    \hfill
    \begin{minipage}[b]{0.31\textwidth}
        \centering
        \subfloat[Case 9 for $\widetilde{C}_2$]{%
        \label{fig:C2case9}
        \resizebox{\linewidth}{!}{%
            \begin{tikzpicture}[shorten >=1pt, baseline=(current bounding box.center)]
            \fill[black!20] (0,0) -- (0,1) -- (1,1) -- cycle;
            \fill[black!20] (0,5) -- (1,5) -- (0,6) -- cycle;
            \fill[black!20] (0,6) -- (1,6) -- (1,5) -- cycle;
            \fill[black!20] (9,5) -- (10,6) -- (10,5) -- cycle;
            \draw (0,6) -- (0,0) -- (6,0) -- (10,4) -- (10,6) -- cycle;
            \draw (0,0) -- (0,1) -- (1,1) -- cycle;
            \draw (0,5) -- (1,5) -- (0,6) -- cycle;
            \draw (0,6) -- (1,6) -- (1,5) -- cycle;
            \draw (9,5) -- (10,6) -- (10,5) -- cycle;
            \draw[dashed] (9,5) -- (10,4);
            \node at (0.3,0.7) {$u$};
            \node at (0.3,5.3) {$v$};
            \node at (0.7,5.7) {$v$};
            \node at (9.7,5.3) {$w$};
            \end{tikzpicture}%
        }}
    \end{minipage}

    \caption{All reduced cases for $\widetilde{C}_2$}
    \label{fig:all_C2_cases}
\end{figure}

In Figs.~\ref{fig:C2case1}--\ref{fig:C2case3}, the orientation of an interior chamber $v$ is arbitrary among the four permissible directions. Generally, strategic reductions and rotations transform any configuration into the canonical cases in Fig.~\ref{fig:all_C2_cases}. It thus suffices to verify the strong hull inequality~\eqref{strong hull property} for these forms. We focus on the most intricate scenario, Fig.~\ref{fig:C2case2}; to explicitly demonstrate orientation independence, we analyze a variation (Fig.~\ref{C2case2'}) where $v$ adopts a distinct orientation. The verification methodology remains invariant regardless of such choices.

    \begin{figure}[htbp]
    \centering

    \begin{tikzpicture}[shorten >=1pt, node distance=2cm, auto]
            
            \fill[black!20] (0,1) -- (1,0) -- (1,1) -- cycle;
            \fill[black!20] (9,5) -- (10,6) -- (10,5) -- cycle;
            \fill[black!20] (5,2) -- (6,2) -- (5,3) -- cycle;
            \fill[black!20, opacity=0.3] (0,1) -- (1,0) -- (4,0) -- (6,2) -- (5,3) -- (2,3) -- cycle;
            \fill[black!20, opacity=0.3] (5,3) -- (5,2) -- (6,2) -- (8,2) -- (10,4) -- (10,6) -- (8,6) -- cycle;
            \draw (0,1) -- (1,0) -- (6,0) -- (10,4) -- (10,6) -- (5,6) -- cycle;
            \draw (0,1) -- (1,0) -- (1,1) -- cycle;
            \draw (9,5) -- (10,6) -- (10,5) -- cycle;
            \draw (5,2) -- (6,2) -- (5,3) -- cycle;
            \draw[dashed] (9,5) -- (10,4);
            \node at (0.7,0.7) {$u$};
            \node at (9.7,5.3) {$w$};
            \node at (5.3,2.3) {$v$};

        \end{tikzpicture}%
        
    \caption{A case for $\widetilde{C}_2$ of Fig.~\ref{fig:C2case2}}
    \label{C2case2'}
    \end{figure}

Let $v$ correspond to the chamber located at the $a$-th position from left to right in the $b$-th row (counting from bottom upwards) of the convex hull in Fig.~\ref{C2case2'}, and let $w$ correspond to the chamber at the $a$-th position in the $y$-th row. For $v$ and $w$ to align with the orientations shown in Fig.~\ref{C2case2'}, the congruence conditions $a \equiv 2 \pmod{4}$ and $x \equiv 1 \pmod{4}$ must hold. Furthermore, the inequalities $x \geq a + 3$ and $y > b \geq 2$ are required to ensure the convex hull formed by $u$, $v$, and $w$ matches the configuration in Fig.~\ref{C2case2'}.

Following the analytical framework for $\widetilde{A}_2$ in Section~\ref{affine type A2}, we calculate the cardinalities of chambers in the respective convex hulls:
\begin{equation}
    \begin{aligned}
        |\mathrm{Conv}(u,v)|&= a + (b-2)(a+2) + a\\
        &= 2a + (b-2)(a+2),
    \end{aligned}
\end{equation}
for the convex hull of $u$ and $v$;
\begin{equation}
    \begin{aligned}
        |\mathrm{Conv}(v,w)| &= (x-a+4) + (y-b-2)(x-a+5) + (x-a+4) + (x-a+2)\\
        &= 3(x-a+3) + (y-b-2)(x-a+5),
    \end{aligned}
\end{equation}
for the convex hull of $v$ and $w$; and
\begin{equation}
    \begin{aligned}
        |\mathrm{Conv}(u,v,w)|&= (x+1) + (x+2) + (y-3)(x+3) + x\\
        &= 3(x+1) + (y-3)(x+3),
    \end{aligned}
\end{equation}
for the combined convex hull. The verification of the strong hull inequality~\eqref{strong hull property} under these configurations reduces to proving Proposition~\ref{finalprop trans C2}, as outlined in the below.

\begin{proposition}\label{finalprop trans C2}
    Let positive integers $a$, $b$, $x$, $y$ satisfy:
    \begin{enumerate}[(i)]
        \item $a \equiv 2 \pmod{4}$ and $x \equiv 1 \pmod{4}$,
        \item $y > b \geq 2$,
        \item $x \geq a + 3$.
    \end{enumerate}
    Then the inequality
    \begin{equation}\label{finalineq trans C2}
        \begin{aligned}
        \big[2a &+ (b - 2)(a + 2)\big]\big[3(x - a + 3) + (y - b - 2)(x - a + 5)\big] \\
        &\geq 3(x + 1) + (y - 3)(x + 3)
        \end{aligned}
    \end{equation}
    holds under the given constraints.
\end{proposition}

\begin{proof}
Let $a=4n+2$ and $x=4m+1$ with integers $n$, $m \geq 0$. Under the condition (iii), we have $m \geq n+1$, which allows us to set $m = n + q + 1$ where $q \geq 0$. Similarly, define $b = k + 2$ and $y = k + p + 3$ with $k,p \geq 0$. This yields
\begin{equation}
    x = 4(n + q + 1) + 1 = 4n + 4q + 5.
\end{equation}

To establish the target inequality~\eqref{finalineq trans C2}, we compute the difference between the left-hand side and the right-hand side by substituting the parameterized variables:
\begin{equation}\label{Finalineq prop C2}
    \begin{aligned}
        \mathrm{LHS} - \mathrm{RHS} &= 16knpq+32knp+32knq+36kn+32npq+64nq \\
        &\quad +60np+68n+16kpq+32kp+28kq+32k \\
        &\quad +12pq+24p+20q+22.
    \end{aligned}
\end{equation}

Observe that all coefficients in the polynomial~\eqref{Finalineq prop C2} are strictly positive, while all variables $k,n,p,q$ are non-negative integers. This immediately implies $\mathrm{LHS} - \mathrm{RHS} \geq 0$, and consequently $\mathrm{LHS} \geq \mathrm{RHS}$ as required.
\end{proof}

The remaining cases of type $\widetilde{C}_2$ configurations in Fig.~\ref{fig:all_C2_cases} have been systematically verified through analogous computational procedures, following the established methodology detailed in the preceding arguments. We have thus established Conjecture~\ref{strong hull conj} for Coxeter groups of type $\widetilde{C}_2$.

\begin{theorem}[Strong hull property for type $\widetilde{C}_2$]\label{C2 strong hull property}
    The Cayley graph of affine type $\widetilde{C}_2$ has the strong hull property.
\end{theorem}

\subsection{Affine type \texorpdfstring{$\widetilde{G}_2$}{G2}}

The Coxeter graph for type $\widetilde{G}_2$ is given in Tab.~\ref{Affine irreducible Coxeter groups}. While direct analysis of its Cayley graph appears intractable, the dual reflection hyperplane arrangement (Fig.~\ref{reflection hyperplane of G2}) offers crucial simplifications. Although $\widetilde{G}_2$ exhibits reduced symmetry compared to $\widetilde{A}_2$, making a direct classification strategy (analogous to Sections~\ref{affine type A2} and~\ref{affine type C2}) arduous, this asymmetry affords structural flexibility. By exploiting this, we derive the strong hull property for $\widetilde{G}_2$ by leveraging the established results for $\widetilde{A}_2$. Specifically, verification of Example~\ref{G2 example 1} demonstrates that the result for $\widetilde{G}_2$ arises as a direct corollary of the $\widetilde{A}_2$ case.

\begin{example}\label{G2 example 1}
    As illustrated in Fig.~\ref{fig.G2 example 1}, the chambers corresponding to elements $u$, $v$, and $w$ are positioned at the reflection hyperplanes of type $\widetilde{G}_2$. The convex hull $\mathrm{Conv}(u,v,w)$ is bounded by thin solid lines, while $\mathrm{Conv}(u,v)$ and $\mathrm{Conv}(v,w)$ are highlighted with shaded regions. Notably, the chambers associated with $u$ and $v$ may reside either entirely within a regular triangle or have a nearest regular triangle in proximity. We select the regular triangle that lies strictly inside the convex hull generated by $u$ and $v$, which is emphasized by thick solid lines in Fig.~\ref{fig.G2 example 1}.
    
    \begin{figure}[htbp]
    \centering
    \resizebox{\textwidth}{!}{

    \begin{tikzpicture}[shorten >=1pt, node distance=2cm, auto]
            
            \fill[black!20] (0,0) -- (1,0) -- (0,{sqrt(3)}) -- cycle;
            \fill[black!20] (-3,0) -- (-3,{sqrt(3)}) -- (-4,{sqrt(3)}) -- cycle;
            \fill[black!20] (2,{7*sqrt(3)}) -- (3,{7*sqrt(3)}) -- (3,{6*sqrt(3)}) -- cycle;
            \fill[black!20] (12,{3*sqrt(3)}) -- (14,{3*sqrt(3)}) -- (13.5,{3.5*sqrt(3)}) -- cycle;
            \fill[black!20] (15,{6*sqrt(3)}) -- (17,{6*sqrt(3)}) -- (16.5,{6.5*sqrt(3)}) -- cycle;
            \fill[black!20, opacity=0.3] (3,{7*sqrt(3)}) -- (2,{7*sqrt(3)}) -- (0,{5*sqrt(3)}) -- (0,0) -- (1,0) -- (3,{2*sqrt(3)}) -- cycle;
            \fill[black!20, opacity=0.3] (2,{7*sqrt(3)}) -- (3,{6*sqrt(3)}) -- (13.5,{2.5*sqrt(3)}) -- (14,{3*sqrt(3)}) -- (12,{5*sqrt(3)}) -- (6,{7*sqrt(3)}) -- cycle;
            \fill[black!20, opacity=0.3] (-4,{sqrt(3)}) -- (-3,0) -- (3,{6*sqrt(3)}) -- (3,{7*sqrt(3)}) -- (2,{7*sqrt(3)}) -- cycle;
            \fill[black!20, opacity=0.3] (2,{7*sqrt(3)}) -- (3,{6*sqrt(3)}) -- (17,{6*sqrt(3)}) -- (16,{7*sqrt(3)}) -- cycle;
            
            \draw (0,0) -- (1,0) -- (0,{sqrt(3)}) -- cycle;
            \draw (-3,0) -- (-3,{sqrt(3)}) -- (-4,{sqrt(3)}) -- cycle;
            \draw (2,{7*sqrt(3)}) -- (3,{7*sqrt(3)}) -- (3,{6*sqrt(3)}) -- cycle;
            \draw (12,{3*sqrt(3)}) -- (14,{3*sqrt(3)}) -- (13.5,{3.5*sqrt(3)}) -- cycle;
            \draw (15,{6*sqrt(3)}) -- (17,{6*sqrt(3)}) -- (16.5,{6.5*sqrt(3)}) -- cycle;
            \draw (-4,{sqrt(3)}) -- (-3,0) -- (3,0) -- (15,{4*sqrt(3)}) -- (17,{6*sqrt(3)}) -- (16,{7*sqrt(3)}) -- (2,{7*sqrt(3)}) -- cycle;
            \draw (0,{5*sqrt(3)}) -- (0,0) -- (9,0) -- (12,{sqrt(3)}) -- (14,{3*sqrt(3)}) -- (12,{5*sqrt(3)}) -- (6,{7*sqrt(3)}) -- (2,{7*sqrt(3)}) -- cycle;
            
            \draw[dashed] (-4,0) -- (17,0);
            \draw[dashed] (-4,{sqrt(3)}) -- (17,{sqrt(3)});
            \draw[dashed] (-4,{2*sqrt(3)}) -- (17,{2*sqrt(3)});
            \draw[dashed] (-4,{3*sqrt(3)}) -- (17,{3*sqrt(3)});
            \draw[dashed] (-4,{4*sqrt(3)}) -- (17,{4*sqrt(3)});
            \draw[dashed] (-4,{5*sqrt(3)}) -- (17,{5*sqrt(3)});
            \draw[dashed] (-4,{6*sqrt(3)}) -- (17,{6*sqrt(3)});
            \draw[dashed] (-4,{7*sqrt(3)}) -- (17,{7*sqrt(3)});
            \draw[dashed] (-3,0) -- (-3,{7*sqrt(3)});
            \draw[dashed] (0,0) -- (0,{7*sqrt(3)});
            \draw[dashed] (3,0) -- (3,{7*sqrt(3)});
            \draw[dashed] (6,0) -- (6,{7*sqrt(3)});
            \draw[dashed] (9,0) -- (9,{7*sqrt(3)});
            \draw[dashed] (12,0) -- (12,{7*sqrt(3)});
            \draw[dashed] (15,0) -- (15,{7*sqrt(3)});
            \draw[dashed] (-4,{sqrt(3)}) -- (-3,0);
            \draw[dashed] (-4,{3*sqrt(3)}) -- (-1,0);
            \draw[dashed] (-4,{5*sqrt(3)}) -- (1,0);
            \draw[dashed] (-4,{7*sqrt(3)}) -- (3,0);
            \draw[dashed] (-2,{7*sqrt(3)}) -- (5,0);
            \draw[dashed] (0,{7*sqrt(3)}) -- (7,0);
            \draw[dashed] (2,{7*sqrt(3)}) -- (9,0);
            \draw[dashed] (4,{7*sqrt(3)}) -- (11,0);
            \draw[dashed] (6,{7*sqrt(3)}) -- (13,0);
            \draw[dashed] (8,{7*sqrt(3)}) -- (15,0);
            \draw[dashed] (10,{7*sqrt(3)}) -- (17,0);
            \draw[dashed] (12,{7*sqrt(3)}) -- (17,{2*sqrt(3)});
            \draw[dashed] (14,{7*sqrt(3)}) -- (17,{4*sqrt(3)});
            \draw[dashed] (16,{7*sqrt(3)}) -- (17,{6*sqrt(3)});
            \draw[dashed] (-4,{5*sqrt(3)}) -- (-2,{7*sqrt(3)});
            \draw[dashed] (-4,{3*sqrt(3)}) -- (0,{7*sqrt(3)});
            \draw[dashed] (-4,{sqrt(3)}) -- (2,{7*sqrt(3)});
            \draw[dashed] (-3,0) -- (4,{7*sqrt(3)});
            \draw[dashed] (-1,0) -- (6,{7*sqrt(3)});
            \draw[dashed] (1,0) -- (8,{7*sqrt(3)});
            \draw[dashed] (3,0) -- (10,{7*sqrt(3)});
            \draw[dashed] (5,0) -- (12,{7*sqrt(3)});
            \draw[dashed] (7,0) -- (14,{7*sqrt(3)});
            \draw[dashed] (9,0) -- (16,{7*sqrt(3)});
            \draw[dashed] (11,0) -- (17,{6*sqrt(3)});
            \draw[dashed] (13,0) -- (17,{4*sqrt(3)});
            \draw[dashed] (15,0) -- (17,{2*sqrt(3)});
            \draw[dashed] (-4,{sqrt(3)/3}) -- (-3,0);
            \draw[dashed] (-4,{2*sqrt(3)+sqrt(3)/3}) -- (3,0);
            \draw[dashed] (-4,{4*sqrt(3)+sqrt(3)/3}) -- (9,0);
            \draw[dashed] (-4,{6*sqrt(3)+sqrt(3)/3}) -- (15,0);
            \draw[dashed] (0,{7*sqrt(3)}) -- (17,{sqrt(3)+sqrt(3)/3});
            \draw[dashed] (6,{7*sqrt(3)}) -- (17,{3*sqrt(3)+sqrt(3)/3});
            \draw[dashed] (12,{7*sqrt(3)}) -- (17,{5*sqrt(3)+sqrt(3)/3});
            \draw[dashed] (-4,{6*sqrt(3)-sqrt(3)/3}) -- (0,{7*sqrt(3)});
            \draw[dashed] (-4,{4*sqrt(3)-sqrt(3)/3}) -- (6,{7*sqrt(3)});
            \draw[dashed] (-4,{2*sqrt(3)-sqrt(3)/3}) -- (12,{7*sqrt(3)});
            \draw[dashed] (-3,0) -- (17,{7*sqrt(3)-sqrt(3)/3});
            \draw[dashed] (3,0) -- (17,{5*sqrt(3)-sqrt(3)/3});
            \draw[dashed] (9,0) -- (17,{3*sqrt(3)-sqrt(3)/3});
            \draw[dashed] (15,0) -- (17,{sqrt(3)-sqrt(3)/3});
            
            \draw[line width=4pt] (-4,{sqrt(3)}) -- (-3,0) -- (-2,{sqrt(3)}) -- cycle;
            \draw[line width=4pt] (0,{sqrt(3)}) -- (1,0) -- (2,{sqrt(3)}) -- cycle;
            \draw[line width=4pt] (12,{3*sqrt(3)}) -- (13,{4*sqrt(3)}) -- (14,{3*sqrt(3)}) -- cycle;
            \draw[line width=4pt] (15,{6*sqrt(3)}) -- (16,{7*sqrt(3)}) -- (17,{6*sqrt(3)}) -- cycle;
            \draw[line width=4pt] (2,{7*sqrt(3)}) -- (4,{7*sqrt(3)}) -- (3,{6*sqrt(3)}) -- cycle;
            
            \draw[dashed, line width=4pt, dash pattern=on 8pt off 4pt] (-4,{sqrt(3)}) -- (-2,{sqrt(3)}) -- (3,{6*sqrt(3)}) -- (1,{6*sqrt(3)}) -- cycle;
            \draw[dashed, line width=4pt, dash pattern=on 8pt off 4pt] (0,{sqrt(3)}) -- (2,{sqrt(3)}) -- (3,{6*sqrt(3)}) -- (1,{6*sqrt(3)}) -- cycle;
            \draw[dashed, line width=4pt, dash pattern=on 8pt off 4pt] (3,{6*sqrt(3)}) -- (4,{7*sqrt(3)}) -- (16,{7*sqrt(3)}) -- (15,{6*sqrt(3)}) -- cycle;
            \draw[dashed, line width=4pt, dash pattern=on 8pt off 4pt] (3,{6*sqrt(3)}) -- (4,{7*sqrt(3)}) -- (13,{4*sqrt(3)}) -- (12,{3*sqrt(3)}) -- cycle;
            
            \fill[pattern=sparse bold dots, opacity=0.3] (-4,{sqrt(3)}) -- (-2,{sqrt(3)}) -- (3,{6*sqrt(3)}) -- (1,{6*sqrt(3)}) -- cycle;
            \fill[pattern=sparse bold dots, opacity=0.3] (0,{sqrt(3)}) -- (2,{sqrt(3)}) -- (3,{6*sqrt(3)}) -- (1,{6*sqrt(3)}) -- cycle;
            \fill[pattern=sparse bold lines, opacity=0.3] (3,{6*sqrt(3)}) -- (4,{7*sqrt(3)}) -- (16,{7*sqrt(3)}) -- (15,{6*sqrt(3)}) -- cycle;
            \fill[pattern=sparse bold lines, opacity=0.3] (3,{6*sqrt(3)}) -- (4,{7*sqrt(3)}) -- (13,{4*sqrt(3)}) -- (12,{3*sqrt(3)}) -- cycle;
            
            \node at ({1/3},{sqrt(3)/3}) {$u$};
            \node at ({-3-1/3},{sqrt(3)-sqrt(3)/3}) {$u_1$};
            \node at ({3-1/3},{7*sqrt(3)-sqrt(3)/3}) {$v$};
            \node at ({14-5/6},{3*sqrt(3)+sqrt(3)/6}) {$w$};
            \node at ({17-5/6},{6*sqrt(3)+sqrt(3)/6}) {$w_1$};

        \end{tikzpicture}%
        }
    \caption{An example of the convex hull in affine type $\widetilde{G}_2$}
    \label{fig.G2 example 1}
    \end{figure}
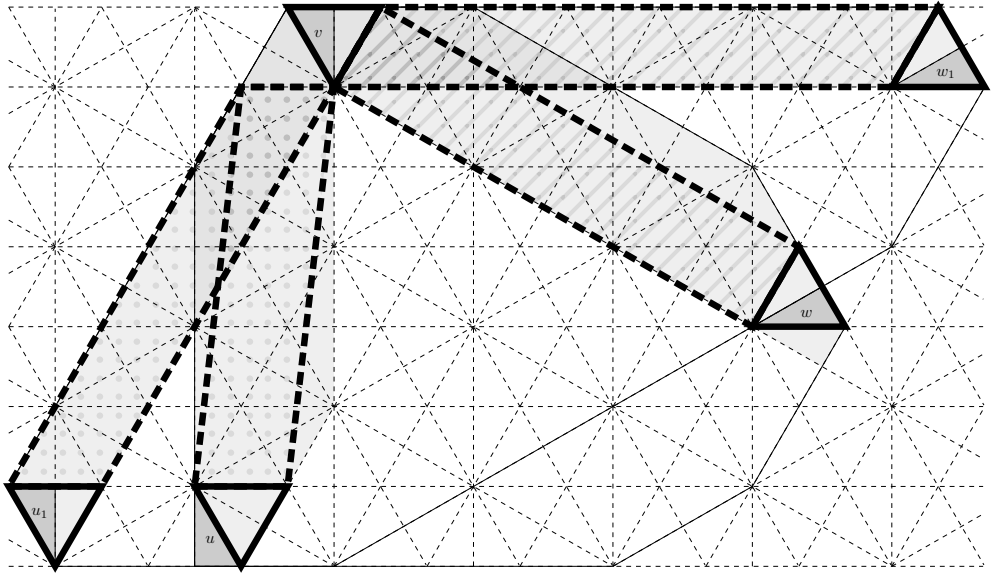

    The obtained regular triangle is then translated along the boundary of $\mathrm{Conv}(u,v,w)$ to the designated position following the direction shown in Fig.~\ref{fig.G2 example 1}. We denote the chambers located farther from $v$ in the resulting regular triangle as $u_1$ and $v_1$. Applying the method from Section~\ref{affine type C2} to draw construction lines, we shade the resulting parallelogram. Each chamber is still treated as a unit area, e.g., $\mathrm{Area}(u)=1$. First analyzing the two parallelograms shaded with dots that share equal areas, we notice that $\mathrm{Area}(P_u)=\mathrm{Area}(P_{u_1})$. Observing that the regular triangle has $\mathrm{Area}(T)=2$, the symmetry of $\mathrm{Conv}(u,v)$ yields the inequality:
\begin{equation}\label{G2 example 1 ineq1}
    |\mathrm{Conv}(u_1,v)|=5+\mathrm{Area}(P_{u_1})\leq \mathrm{Area}(u)+5+\mathrm{Area}(P_{u})+1\leq |\mathrm{Conv}(u,v)|.
\end{equation}
Similarly, $|\mathrm{Conv}(u_1,v)|\leq|\mathrm{Conv}(u,v)|$ holds. Through this process, the relationship between $|\mathrm{Conv}(u_1,v,w_1)|$ and $|\mathrm{Conv}(u,v,w)|$ cannot be directly discerned. Temporarily disregarding the vertical walls and 30°-inclined walls in the reflection hyperplanes of type $\widetilde{G}_2$ in Fig.~\ref{fig.G2 example 1}, we effectively work within reflection hyperplanes of type $\widetilde{A}_2$. Under this framework, each group element resides in a new $\widetilde{A}_2$-chamber. For instance, $u_1$ lies in the bold-outlined regular triangle at the lower left. Let $\mathrm{Conv}_{\widetilde{A}_2}(u,v,w)$ denote the convex hull generated in $\widetilde{A}_2$. Returning to the $\widetilde{G}_2$ reflection hyperplanes, we establish:
\begin{equation}\label{G2 example 1 ineq2}
    2|\mathrm{Conv}_{\widetilde{A}_2}(u_1,v,w_1)|\geq\max\{\mathrm{Conv}(u,v,w),\mathrm{Conv}(u_1,v,w_1)\}.
\end{equation}
Given
\begin{equation}
    |\mathrm{Conv}(u_1,v)|\geq 2|\mathrm{Conv}_{\widetilde{A}_2}(u_1,v)|-1
\end{equation} 
and
\begin{equation}
    |\mathrm{Conv}(v,w_1)|\geq 2|\mathrm{Conv}_{\widetilde{A}_2}(v,w_1)|-1,
\end{equation}
where $|\mathrm{Conv}_{\widetilde{A}_2}(u_1,v)|$, $|\mathrm{Conv}_{\widetilde{A}_2}(v,w_1)|\geq 2$, we have
\begin{equation}
    \begin{aligned}
        &|\mathrm{Conv}(u_1,v)|\cdot|\mathrm{Conv}(v,w_1)|\\
        \geq&\left(2|\mathrm{Conv}_{\widetilde{A}_2}(u_1,v)|-1\right)\left(2|\mathrm{Conv}_{\widetilde{A}_2}(v,w_1)|-1\right)\\
        =&4|\mathrm{Conv}_{\widetilde{A}_2}(u_1,v)|\cdot|\mathrm{Conv}_{\widetilde{A}_2}(v,w_1)|-2\left(|\mathrm{Conv}_{\widetilde{A}_2}(u_1,v)|+|\mathrm{Conv}_{\widetilde{A}_2}(v,w_1)|\right)+1\\
        \geq&2|\mathrm{Conv}_{\widetilde{A}_2}(u_1,v)|\cdot|\mathrm{Conv}_{\widetilde{A}_2}(v,w_1)|.
    \end{aligned}
\end{equation}
It follows from Theorem~\ref{A2 strong hull property} that
\begin{equation}
    |\mathrm{Conv}_{\widetilde{A}_2}(u_1,v)|\cdot|\mathrm{Conv}_{\widetilde{A}_2}(v,w_1)|\geq |\mathrm{Conv}_{\widetilde{A}_2}(u_1,v,w_1)|.
\end{equation}
\end{example}

For Coxeter groups of type $\widetilde{G}_2$, the reduction techniques developed in Example~\ref{G2 example 1} allow us to reduce convex hulls to the $\widetilde{A}_2$ setting, where the required property has already been established through Theorem~\ref{A2 strong hull property}. Since any chamber corresponding to an element in type $\widetilde{G}_2$ is transformed into a chamber of regular triangular shape within type $\widetilde{A}_2$, the same procedure applies. We have thus established Conjecture~\ref{strong hull conj} for Coxeter groups of type $\widetilde{G}_2$.

\begin{theorem}\label{G2 strong hull property}
    The Cayley graph of affine type $\widetilde{G}_2$ has the strong hull property.
\end{theorem}

The combined results of Theorems~\ref{A2 strong hull property}, \ref{C2 strong hull property}, and \ref{G2 strong hull property} collectively establish the main result of this paper, Theorem~\ref{main result}.

\section*{Acknowledgments}
The author wishes to thank his bachelor's thesis advisor Shoumin Liu for his supervision. The author is also grateful to his mentor Yibo Gao for suggesting the problem regarding type $\widetilde{A}$ during the 2024 PKU Algebra and Combinatorics Experience. Special thanks are due to Shiyu Xiu for valuable discussions and for sharing his geometric insights regarding the type $\widetilde{A}_2$, which provided crucial inspiration for the reduction techniques developed in this paper. This work was partially supported by NSFC grant 12426507.

\bibliographystyle{amsplain}
\bibliography{sample_2025}

\end{document}